\documentclass[11pt,twoside]{amsart}
\usepackage{amssymb}
\usepackage{amsmath}
\usepackage[matrix,arrow]{xy}
\usepackage{graphics}
\usepackage{palatino}
\usepackage{pgf}
\usepackage{pdflscape}
\usepackage{afterpage}
\usepackage{capt-of}
\usepackage{multirow}
\usepackage{xcolor}
\usepackage{enumerate}

\theoremstyle{plain}
\newtheorem{prop}{Proposition}[section]
\newtheorem{lemma}[prop]{Lemma}
\newtheorem{thm}[prop]{Theorem}
\newtheorem*{MainThm}{Main Theorem}
\newtheorem{cor}[prop]{Corollary}

\theoremstyle{remark}
\newtheorem{rmk}[prop]{Remark}
\theoremstyle{definition}
\newtheorem{defn}[prop]{Definition}
\newtheorem{ex}[prop]{Example}

\renewcommand{\AA}{\mathbb{A}}

\newcommand{\CC}{\mathbb{C}}

\newcommand{\FF}{\mathbb{F}}
\newcommand{\GG}{\mathbb{G}}

\newcommand{\NN}{\mathbb{N}}

\newcommand{\PP}{\mathbb{P}}
\newcommand{\QQ}{\mathbb{Q}}
\newcommand{\RR}{\mathbb{R}}

\newcommand{\ZZ}{\mathbb{Z}}

\newcommand{\MGL}{\mathsf{MGL}}
\newcommand{\BP}{\mathsf{BP}}
\newcommand{\BPn}[1]{\mathsf{BP}\langle {#1} \rangle}

\newcommand{\KGL}{\mathsf{KGL}}
\newcommand{\SH}{\mathsf{SH}}

\newcommand{\EEE}{\mathsf{E}}
\newcommand{\FFF}{\mathsf{F}}
\newcommand{\GGG}{\mathsf{G}}

\newcommand{\HHPP}{\mathbb{HP}}

\newcommand{\M}{\mathsf{M}}
\newcommand{\HHH}{\mathsf{H}}
\newcommand{\one}{\mathsf{S}}

\newcommand{\smsh}{\wedge}
\newcommand{\Ext}{\operatorname{Ext}}

\newcommand{\Spec}{\operatorname{Spec}}
\newcommand{\MZ}{\M\ZZ}
\newcommand{\cd}{\operatorname{cd}}
\renewcommand{\top}{{\operatorname{top}}}
\newcommand{\eff}{{\operatorname{eff}}}
\newcommand{\chr}{\operatorname{char}}
\newcommand{\Tor}{\operatorname{Tor}}
\newcommand{\KO}{\mathsf{KO}}

\newcommand{\Sq}{\operatorname{Sq}}
\newcommand{\red}{\operatorname{red}}

\newcommand{\hp}{\mathsf{hp}}
\newcommand{\TOP}{\mathsf{top}}
\newcommand{\uhom}{\underline{\mathrm{Hom}}}
\newcommand{\op}{\mathsf{op}}

\newcommand{\comp}[1]{^{\smsh}_{#1}}

\usepackage[margin=1.5in]{geometry}

\begin{document}

\title{Stable motivic $\pi_1$ of low-dimensional fields}

\author{Kyle M.~Ormsby}
\address{Department of Mathematics, MIT, USA}
\email{ormsby@math.mit.edu}

\author{Paul Arne {\O}stv{\ae}r}
\address{Department of Mathematics, University of Oslo, Norway}
\email{paularne@math.uio.no}


\begin{abstract}
Let $k$ be a field with cohomological dimension less than $3$; 
we call such fields \emph{low-dimensional}.  
Examples include algebraically closed fields, finite fields and function fields thereof, 
local fields, 
and number fields with no real embeddings.
We determine the $1$-column of the motivic Adams-Novikov spectral sequence over $k$.
Combined with rational information we use this to compute $\pi_1\one$, 
the first stable motivic homotopy group of the sphere spectrum over $k$.  
Our main result affirms Morel's $\pi_1$-conjecture in the case of low-dimensional fields.
We also determine $\pi_{1+n\alpha}\one$ for weights $n\in \ZZ\smallsetminus\{-2,-3,-4\}$.  
\end{abstract}
\maketitle

\section{Introduction} \label{sec:intro}

The stable motivic homotopy groups of the sphere spectrum over a field form an
interesting and computationally challenging class of invariants.  
They are at least as difficult to compute as the stable homotopy groups of the topological 
sphere spectrum (cf.~\cite{FullFaithful} for a precise sense in which this is true),
but also incorporate a great deal of nuanced arithmetic information
about the base field.  These groups were first explored by 
Morel \cite{MorelMASS,Morelpi0}, who computed the 0-th group as the Grothendieck-Witt ring of
quadratic forms.  Work of Dugger-Isaksen \cite{DI} and Hu-Kriz-Ormsby \cite{HKO} 
produced hands-on computations in a range of dimensions over
algebraically closed fields, but results encompassing a richer class
of fields remained elusive.  

Given a field $k$, let $p$ denote the exponential characteristic
of $k$ ($p=1$ if $\chr k = 0$; otherwise $p=\chr k$), and let 
$\one[1/p]$ denote the motivic sphere spectrum over $k$ with $p$
inverted,
\[
  \one[1/p] = \operatorname{hocolim}(\one\xrightarrow{p}\one\xrightarrow{p}\cdots).
\]
In particular, if $p=1$ then $\one[1/p] = \one$.  In this paper we determine the first
stable motivic homotopy group of $\one[1/p]$, $\pi_1 \one[1/p]$, over
any field
$k$ with cohomological
dimension less than $3$.  In fact, when $k$ satisfies these hypotheses
and $p\ne 2,3$, we prove a variant of Morel's conjecture on $\pi_1
\one$ stating that there is a short exact sequence
\[
0
\to 
K^M_2(k)/24 
\to 
\pi_1 \one[1/p] 
\to 
K^M_1(k)/2 \oplus \ZZ/2
\to 
0.
\]
Here, 
$K^M_*(k)$ denotes Milnor $K$-theory, 
defined by generators and relations in \cite{MilnorK}.
The short exact sequence for $\pi_1 \one[1/p]$ is a stable version of \cite[Conjecture 7]{Asok2}.

Our cohomological dimension assumption holds for many examples of interest, 
e.g., 
algebraically closed fields, finite fields, local fields, and number fields with no real embeddings \cite{NSW}.
The factor of $24$ is related to the fact that $\pi_3$ of the topological sphere spectrum is cyclic of that order, 
and also to the computation of unstable $\pi_{3}(\AA^3\smallsetminus 0)$ by Asok-Fasel in their work on 
splittings of vector bundles \cite{AF2}.

To put our result in context,
recall that Morel-Voevodsky \cite{MorelVoevodsky} constructed a homotopy theory of smooth $k$-schemes 
in which the affine line, 
$\AA^1$, 
plays the r\^{o}le of the unit interval; 
this is called motivic homotopy theory.  
Whenever we have a homotopy theory of a particular type of spaces, 
we may invert the smash product with a space $T$ by forming a category of $T$-spectra.  
In motivic homotopy theory, 
inverting $T=\PP^1$ corresponds to inverting the Lefschetz motive, 
and the resulting homotopy theory of $\PP^1$-spectra is called stable
motivic homotopy theory.  There are several Quillen equivalent model category
structures on $\PP^1$-spectra giving the same stable homotopy theory.
For concreteness and convenience, we will work with
Bousfield-Friedlander $\PP^1$-spectra as in \cite[Theorem 2.9]{MSS}.

In order to detect weak equivalences of $\PP^1$-spectra, one applies
the \emph{homotopy presheaf functor}, $\underline{\pi}_\star$.  
Here $\star$ stands for all indices of the form $(m,n)\in \ZZ^2$.
Following a grading convention inspired by $\ZZ/2$-equivariant
homotopy theory, we write $m+n\alpha$ for $(m,n)$, and let
$S^{m+n\alpha} = (S^1)^{\smsh m} \smsh (\GG_m)^{\smsh n}$.  
With this definition we can identify $\PP^1$ with the
$(1+\alpha)$-sphere in the motivic homotopy category.
Then for a $\PP^1$-spectrum $\EEE$, $\underline{\pi}_{m+n\alpha}(\EEE)$ 
is the presheaf taking a smooth $k$-scheme $X$ to
\[
(\underline{\pi}_{m+n\alpha}\EEE)(X) 
= 
[S^{m+n\alpha}\smsh X_+,\EEE],
\]
where $X_+$ denotes $X$ with a disjoint basepoint and $[-,-]$ denotes
stable motivic homotopy classes of maps.
For cellular motivic spectra \cite{DICell}, 
built from $S^1$ and $\GG_m$ by homotopy pushouts, 
weak equivalences are detected by actual homotopy groups.  
That is, 
we can simply compute  
\[
\pi_{m+n\alpha}\EEE 
= 
(\underline{\pi}_{m+n\alpha}\EEE)(\Spec k).
\]
Thus in order to understand stable motivic homotopy theory, 
it is essential to compute stable motivic homotopy groups.
The ring $\pi_{\star}\one$ is particularly interesting due to the distinguished r\^{o}le of the sphere spectrum.

Expressed in the notation just introduced, 
Morel computed 
\[
\pi_{n\alpha} \one 
= 
K^{MW}_{-n}(k),
\]
where $K^{MW}_*(k)$ is the so-called Milnor-Witt $K$-theory of $k$, 
a graded ring generated as an algebra by symbols $[u]$, 
$u\in k^\times$, 
in degree $1$ and $\eta$ in degree $-1$ subject to the relations
\[
\begin{aligned}
  {}[u][1-u] &= 0, \\
  [uv] &= [u] + [v] + [u][v]\eta,\\
  [u]\eta &= \eta[u], \\
  (2+[-1]\eta)\eta &= 0.
\end{aligned}
\]
This result was first proved for perfect fields of characteristic
unequal to $2$ in \cite[Theorem 6.2.1 and 6.2.2]{Morelpi0}, but in
fact holds for all fields by \cite[Corollary 6.43]{MorelA1} and a base
change argument ala \cite[Appendix C]{AyoubEtale}.  
As a particular case of Morel's theorem, we note that 
\[
  \pi_0\one = GW(k)
\]
where $GW(k)$ is the Grothendieck-Witt ring of symmetric bilinear forms over $k$.  

Note that $\underline{\pi}_{m+n\alpha} \one = 0$ for $m<0$ by Morel's $\AA^1$-connectivity theorem \cite{Morelstableconnectivity}.
Thus the next sensible class of stable motivic homotopy groups to compute is $\pi_{1+*\alpha} \one$.  
It is this task which we undertake in this paper.

It will be convenient to have a short moniker for the class of fields considered in this paper.  
We propose the term \emph{low-dimensional},
though we do not insist on this terminology being used outside these
pages.  We refer the reader to \cite[I \S3 \& II \S2]{Serre} for basic facts
about cohomological dimension.

\begin{defn} \label{defn:lowdim}
A field $k$ is called $\ell$-\emph{low-dimensional} if $\cd_\ell k \le 2$ and $\ell\ne p$.
A field $k$ is \emph{low-dimensional} if $\cd k\le 2$.
\end{defn}

Recall \cite[II \S2.2 Proposition 3]{Serre} that $cd_p k\le 1$
whenever $k$ has positive characteristic $p$, so $k$ is
low-dimensional if and only if $k$ is $\ell$-low-dimensional for all
primes $\ell\ne p$.

Our computations intertwine motivic and topological Hopf maps in nontrivial ways.
Recall the stable motivic Hopf map $\eta$ is represented unstably by the Hopf fibration $\AA^{2}\smallsetminus 0\to\PP^{1}$,
$(x,y)\mapsto [x,y]$.
Its cone is isomorphic to $\PP^{2}$.
The element $\eta\in\pi_{\alpha} \one$ is nontrivial since $\MZ/2^\star(\PP^{2})\cong\MZ/2^\star[t]/(t^3)$ and $\Sq^2(t)=t^2$ on the 
generator $t$ of mod $2$ motivic cohomology in bidegree $1+\alpha$. 
A similar argument with $\Sq^4$ and the quaternionic projective space $\HHPP^{2}$ shows the quaternionic motivic Hopf map 
$\nu\in\pi_{1+2\alpha} \one$ is nontrivial,
cf.~\cite[Theorem 4.17]{AF2} and \cite[Remark 4.14]{DDDI}. 
Finally, 
by triangulating the topological Hopf map $S^3\to S^2$ and applying
the constant presheaf functor, we get the simplicial Hopf map $\eta_s$
in motivic spectra.  
We are now ready to state our main theorem, which determines $\pi_1 \one[1/p]$ as a group.

\begin{MainThm} \label{thm:main}
For a low-dimensional field $k$ with $p\ne 2,3$, 
there is a short exact sequence of abelian groups
\[
  0\to K^M_2(k)/24 \to \pi_1 \one[1/p] \to K^M_1(k)/2 \oplus \ZZ/2\to 0.
\]
If we consider $\pi_{1+*\alpha}\one[1/p]$ as a $\pi_{*\alpha}\one[1/p]
\cong K^{MW}_{-*}(k)[1/p]$-module, then $K^M_2(k)/24 \subset
\pi_1\one[1/p]$ consists of $K^{MW}_2(k)\cong \pi_{-2\alpha}\one$-multiples of
$\nu$.  Moreover, the $K^{MW}_0(k)\cong GW(k)$-multiples of $\eta_s$ map onto
$K^M_1(k)/2 \oplus \ZZ/2$ via $\langle u\rangle \eta_s\mapsto ([u],1)$. 

As a short exact sequence of abelian groups, the $\ZZ/2\{\eta_s\}$-summand splits, 
while for $[u],[v]\in K^M_1(k)/2$ we have
\[
  [u]\eta\eta_s+[v]\eta\eta_s = [uv]\eta\eta_s - 12[u,v]\nu.
\]
\end{MainThm}

In fact, 
we prove more in Theorem \ref{thm:pi1}, 
which addresses the $p=2,3$ cases as well.
For all $p$ our computations are in agreement with Morel's conjecture.
Moreover, 
in \S \ref{sec:nonzero} we produce a brief outline of computations of
the $1$-line $\pi_{1+n\alpha} \one[1/p]$ for $n\in\ZZ\smallsetminus
\{-2,-3,-4\}$.  
When $k$ is a finite field we compute $\pi_{1+n\alpha} \one[1/p] $ in all weights.

\begin{ex}
We illustrate the main theorem with several specific cases.

\begin{itemize}
\item
If $k$ is algebraically closed with $p\ne 2$, 
then $\pi_1 \one[1/p]\cong\ZZ/2\{\eta_s\}$ is the first topological stable stem.
This extends the $\pi_1$ case of rigidity for stable motivic homotopy
to fields of positive characteristics \cite{FullFaithful}.
\item
If $k$ is a finite field with $p\ne 2$,
then $\pi_1 \one[1/p]\cong\ZZ/2\oplus\ZZ/2$.
\item
If $F$ is a finite field and $k=F(T)$ with $p\ne 2,3$, 
there is an exact sequence
\[
0
\to 
\bigoplus_{\mathfrak{p}} (F[T]/\frak{p})^{\times}\otimes\ZZ/24  
\to 
\pi_1 \one[1/p] 
\to 
F(T)^{\times}/2 \oplus \ZZ/2\to 0,
\]
where $\frak{p}$ runs through the maximal ideals of the affine line $F[T]$.
\item
Suppose $k$ is a non-archimedean local field and $n$ is the number of roots of unity contained in $k$. 
The Hilbert symbol of order $n$ defines a surjection $K_{2}(k)\to\mu_{n}$ with uniquely divisible kernel $nK_{2}(k)$ \cite{Merkurjev}.
Let $q$ denote the order of the residue field of $k$.
If $p\neq 2,3$ is positive, 
the first $K$-group $K^M_{1}(k)=k^{\times}=\ZZ\oplus\ZZ/(q-1)\oplus (\ZZ_{p})^{\oplus\NN}$ and there is a short exact sequence
\[
0
\to 
\ZZ/(n,24)
\to 
\pi_1 \one [1/p]
\to  
\ZZ/2\oplus\ZZ/(q-1,2)\oplus\ZZ/2
\to 
0.
\]
If $p=1$, 
$K^M_{1}(k)=k^{\times}=\ZZ\oplus\ZZ/(q-1)\oplus\ZZ/\ell^{a}\oplus (\ZZ_{\ell})^{d}$ where $d=[k:\QQ_{\ell}]$, 
$a\geq 0$ is defined by $\mu_{\ell^{\infty}}(k)=\mu_{\ell^{a}}$, 
and there is a short exact sequence
\[
0
\to 
\ZZ/(n,24)
\to 
\pi_1 \one 
\to  
\ZZ/2\oplus\ZZ/(q-1,2)\oplus\ZZ/(\ell^{a},2)\oplus (\ZZ_{\ell}/2)^{d}\oplus\ZZ/2
\to 
0.
\]
\item
If $k$ is a nonreal number field, 
e.g.,
the Gaussian numbers $\QQ(\sqrt{-1})$,
there is an exact sequence
\[
0
\to 
K^M_2(k)/24
\to 
\pi_1 \one 
\to 
k^{\times}/2 \oplus  \ZZ/2
\to 0.
\]
By Dirichlet's unit theorem, 
the group of square classes of units in $k$ is an infinite dimensional $\ZZ/2$-vector space.
Moreover, 
there is an exact sequence 
\[
0
\to 
K_2(\mathcal{O}_k)
\to
K^M_2(k)
\to
\bigoplus_{\mathfrak{p}}  (\mathcal{O}_k/\frak{p})^{\times}
\to 
0,
\]
where $\frak{p}$ runs through the maximal ideals of the rings of integers $\mathcal{O}_k$.
Thus $\pi_1 \one$ is not a finitely generated abelian group.
\end{itemize}
\end{ex}

\subsection{Relation to other work}
The groups $\pi_{1-n\alpha}\one$ are related to the unstable motivic homotopy groups of punctured affine space $\pi_n(\AA^n\smallsetminus 0)$.  
These latter groups (or, in fact, the Nisnevich sheafifications $\underline{\pi}_n^{\AA^1}(\AA^n\smallsetminus 0)$ of the associated presheaves) 
are of central interest in the Asok-Fasel program \cite{Asok2,AF3} studying splitting of vector bundles.  
Indeed, 
there is a direct system given by $\PP^1$-suspension so that
\[
\underline{\pi}_{1-n\alpha}\one 
\cong 
\operatorname{colim}_m \underline{\pi}^{\AA^1}_{n+m(1+\alpha)} (\PP^1)^{\smsh m} \smsh \AA^n\smallsetminus 0.
\]
This isomorphism follows from \cite[p.~487]{MSS} and the identification
$\AA^n\smallsetminus 0 
\simeq 
S^{n-1+n\alpha}$.
Thus a computation of $\underline{\pi}_{1-n\alpha}\one$ is a computation of the ``stable part'' of Asok-Fasel's obstruction group
$\underline{\pi}_n^{\AA^1}(\AA^n\smallsetminus 0)$.

Unfortunately, it is difficult to make the above statement any more precise.  
In motivic homotopy theory, we lack a $\PP^1$-Freudenthal suspension theorem, 
i.e., 
we lack a theorem giving us a stable range $m\ge N$ for which $\pi_{n+m(1+\alpha)}(\PP^1)^{\smsh m}\smsh \AA^n\smallsetminus 0$ 
is isomorphic to $\underline{\pi}_{1-n\alpha}\one$.  
Producing such a
theorem should be viewed as a significant goal for motivic homotopy
theorists.  It is possible that there would be immediate dividends
in terms of translating the determination of Asok-Fasel's obstruction groups into stable
(and hence more accessible) computations.

Our work is also closely related to (and uses --- see the proof of Proposition
\ref{prop:4nu}) Dugger and Isaksen's recent work on relations between motivic
Hopf elements \cite{DDDI}.  For low-dimensional fields, we verify the
speculated relation $12\nu = \eta^2\eta_s$ from \S 1.7 of \emph{loc. cit.}

\subsection{Outline of the argument}

Our computations proceed one prime at a time via the motivic Adams-Novikov spectral
sequence (MANSS).  In order to use this machinery, we have to
understand certain features of algebraic cobordism and the motivic
Brown-Peterson spectrum over low-dimensional fields.  In
\S\ref{sec:slice}, we use the slice spectral sequence to determine the
coefficient ring $\pi_\star \BP$ and cooperations $\BP_\star\BP$ for
$\BP$ over a low-dimensional field.  In \S\ref{sec:E2}, we use the
special form that the Hopf algebroid $(\pi_\star\BP,\BP_\star\BP)$
takes over low-dimensional fields to determine the $E_2$ page of the
MANSS.
We outline how to use the $\ell$-primary MANSS to compute $\pi_1\one\comp{\ell}$ in the beginning of \S\ref{sec:largePrimePi1},
and then go on to run these computations for $\ell>2$.  
Here $\one\comp{\ell}$ is the $\ell$-completion of the motivic sphere spectrum.  
The crucial $2$-primary computation is contained in \S\ref{sec:2pi1}.

In order to glue our $\ell$-complete computations into an integral
computation, we collect some information about rational motivic
homotopy groups in
\S\ref{sec:rationalPi1}.  That allows us to complete the proof of our 
main theorem in \S\ref{sec:pi1} via arithmetic fracture methods which
we review in Appendix A.  
In \S\ref{sec:nonzero}, we recapitulate our methods for
nonzero weights, completely determining $\pi_{1+n\alpha}\one[1/p]$ for
$n\in \ZZ\smallsetminus \{-2,-3,-4\}$ over low-dimensional fields.  We
compute $\pi_{1+n\alpha}\prod_{\ell\ne p}\one\comp{\ell}$ for all
$n\in \ZZ$ over low-dimensional fields, but some subtle
issues about rational motivic homology groups prevent us from
producing a full answer.

\subsection{Some history}

Many of the ideas in this paper date back to the first author's thesis
\cite{OThesis} but with the extra limitation that the base field $k$
be a $p$-adic field.  
In 2010, the second author adapted these techniques 
to finite fields and local fields of positive
characteristic.  In Fall 2012, 
Aravind Asok gave a talk at Harvard in which he highlighted Morel's
conjecture on $\pi_1 \one$.  This reinvigorated the authors' interest in
the project since it became apparent that machinery developed to
study infinite families in the motivic stable stems could also produce
explicit computations in low degrees.  It was at this point that the authors
sought to generalize their techniques as much as possible and settled
on the low-dimensional fields of Definition \ref{defn:lowdim}.

\subsection{Notation and computational input}

Before continuing, we pause here and
record the notation and computational input used in this paper.  While
some of the following terms have already been defined, it is our hope
that recording everything in one place will assist the reader.

Throughout the subsequent pages,
\begin{itemize}
\item $k$ is a field of exponential characteristic $p$,
\item $\one$ is the motivic sphere spectrum,
\item $S^{m+n\alpha} = (S^1)^{\smsh m}\smsh (\AA^1\smallsetminus
  0)^{\smsh n}$,
\item $\pi_{m+n\alpha}(-) = [S^{m+n\alpha},-]$ where $[-,-]$ denotes
    the hom set in the stable motivic homotopy category over $k$,
\item $\pi_\star = \bigoplus_{m,n\in \ZZ} \pi_{m+n\alpha}$ and
  $\pi_{m+*\alpha} = \bigoplus_{n\in \ZZ} \pi_{m+n\alpha}$,
\item for any motivic spectrum $\EEE$, $\EEE[1/p] = \operatorname{hocolim}
  (\EEE\xrightarrow{p}\EEE\xrightarrow{p}\cdots)$ and $\EEE\comp{\ell}$
  is the $\ell$-completion of $\EEE$,
\item $\M A$ denotes the motivic Eilenberg-MacLane spectrum with
  coefficients in an abelian group $A$,
\item $\KGL$ is the motivic algebraic $K$-theory spectrum,
\item $\KO$ is the motivic Hermitian $K$-theory spectrum,
\item $\MGL$ is the motivic algebraic cobordism spectrum,
\item $\BP$ is the motivic Brown-Peterson spectrum at a prime $\ell$,
  or, in \S 3 and later, the $\ell$-completion of this spectrum,
\item  $K^{MW}_*(k)$ is the $\ZZ$-graded Milnor-Witt $K$-theory ring
  of $k$,
\item $K^M_*(k)$ is the $\NN$-graded Milnor $K$-theory ring of $k$,
\item $\eta$ is the stable motivic Hopf map induced by the projection 
  $S^{1+2\alpha}\simeq \AA^2\smallsetminus 0\to \PP^1$,
\item $\eta_s$ is the stable simplicial Hopf map induced by applying
  the constant presheaf functor to the topological Hopf map,
\item $\nu$ is the stable motivic map induced by applying
  the Hopf construction to the split quaternions, and
\item $\rho$ is the class of $-1$ in $K^{MW}_1(k)\cong
  \pi_{-\alpha}\one$ or the class of $-1$ in $K^M_1(k)\cong
  k^\times\cong \pi_{-\alpha}\M\FF_2$ (the meaning will be clear from
  context).
\end{itemize}

Since they will be important later, we also recall the definitions of
$K^{MW}_*(k)$ and $K^M_*(k)$.  Let $FA(k^\times,\eta)$ denote the free
associative $\ZZ$-graded algebra on symbols $[u]$ in degree $1$ for
$u\in k^\times$ and $\eta$ in degree $-1$.  Let $R^{MW}(k)$ denote the ideal
in $FA(k^\times,\eta)$ generated by $[u][1-u]$ for $u\in
k^\times\smallsetminus 1$, $[uv]-[u]-[v]-[u][v]\eta$ for $u,v\in
k^\times$, $[u]\eta-\eta[u]$ for $u\in k^\times$, and
$(2+\rho\eta)\eta$.  Then
\[
  K^{MW}_*(k) = FA(k^\times,\eta)/R^{MW}(k).
\]

Let $T(k^\times)$ denote the tensor algebra on symbols $[u]$ in degree
$1$ for $u\in
k^\times$.  Let $R^{M}(k)$ denote the ideal in $T(k^\times)$ generated
by $[u][1-u]$ for $u\in k^\times\smallsetminus 1$ and $[uv] - [u] -
[v]$ for $u,v\in k^\times$.  Then
\[
  K^M_*(k) = T(k^\times)/R^M(k)\cong K^{MW}_*(k)/(\eta).
\]

Finally, some arguments in \S\ref{sec:2pi1} will require manipulation
of the coefficient ring $\pi_\star \M\FF_2$ of the mod 2 motivic
Eilenberg-MacLane spectrum $\M\FF_2$.  By the Milnor conjecture, this
ring is isomorphic to $(K^M_*(k)/2)[\tau]$ where $|K^M_n(k)/2| =
-n\alpha$ and $\tau$ is the generator of $\pi_{1-\alpha}\M\FF_2\cong
H^0(k;\mu_2)\cong \{\pm 1\}$.

\subsection{Acknowledgments}
The authors thank Aravind Asok, Mark Behrens, Mike Hopkins, 
Max Karoubi, and Markus Spitzweck for helpful conversations.  The also
thank the anonymous referees for incisive comments and suggestions.

\section{The slice spectral sequence for $\MGL$ and $\BPn{n}$} 
\label{sec:slice}

The slice tower is Voevodsky's substitute for the Postnikov tower in
stable motivic homotopy theory.  We refer to \cite[\S 2]{VOpen} for the rudiments
of its construction, and only briefly recall that, starting with the effective
stable motivic homotopy category $\SH(k)^\eff$, one constructs the
categories $\Sigma_{\PP^1}^q\SH(k)^\eff$ and the adjoint functor pairs
\[
  i_q:\Sigma_{\PP^1}^q\SH(k)^\eff\leftrightarrows \SH(k):r_q
\]
for $q\in \ZZ$.  The $q$-th stage of the slice tower for a motivic spectrum $\EEE$ is
then $f_q\EEE = i_qr_q\EEE$.  By construction these assemble into a tower
\[
\xymatrix{
  \vdots\ar[d] &\\
  f_{q+1}\EEE\ar[d]\ar[r] &s_{q+1}\EEE\\
  f_q\EEE\ar[d]\ar[r] &s_q\EEE\\
  \vdots
}\]
in which $s_q\EEE$, the cofiber of $f_{q+1}\EEE\to f_q\EEE$, is the
$q$-th slice of $\EEE$.

The slice tower for $\EEE$ has an associated tri-graded spectral sequence with
$E_1$-page
\[
  E_1^{q,m+n\alpha} = \pi_{m+n\alpha} s_q\EEE
\]
and differentials
\[
  d_r:E_r^{q,m+n\alpha}\to E_r^{q+r,m-1+n\alpha},
\]
which suggests drawing an ``Adams graded'' picture of the slice
spectral sequence with $q$ on the vertical axis and $m+n\alpha$ on the
horizontal axis (for a fixed weight $n$).

We will freely use the language of spectral sequences in this and the
following sections.  By an infinite cycle, we mean a class $x$ such
that $d_r(x) = 0$ for all $r$.  A permanent cycle is an infinite cycle
which is not a boundary (and hence survives to represent a nonzero
class in $E_\infty$).

In good situations, the $E_1$-page of the slice spectral sequence for
$\EEE$ is computable and the slice spectral sequence converges
to $\pi_\star \EEE$.  This is exactly the case for algebraic
cobordism with $p$ inverted, $\MGL[1/p]$, and the motivic Brown-Peterson spectra $\BP$ at a
prime $\ell\ne p$, where $p$ is the exponential characteristic of $k$.
(For the sake of expediency, we can define $\BP$ as a motivic
Landweber exact spectrum.)  
By the Hopkins-Morel-Hoyois theorem \cite[Theorem 7.12]{Hoyois} and
M.~Spitzweck's relative version of
Voevodksy's conjecture on the slices of $\MGL$ \cite[Theorem 4.7]{SpitzweckRelations}, we have
\[
  s_*\MGL[1/p] = \M\ZZ[1/p][x_1,x_2,\ldots]
\]
where $\M\ZZ[1/p]$ is Voevodsky's integral motivic cohomology spectrum with the
exponential characteristic inverted, and $x_i$ has
dimension $i(1+\alpha)$ and slice degree $i$.  This 
means that $s_q\MGL[1/p]$ is a wedge of $q(1+\alpha)$-suspensions of
$\M\ZZ[1/p]$ indexed by
monomials in the $x_i$ of slice degree $q$. 
Similarly, when $\ell\ne p$ we have
\[
  s_*\BP = \M\ZZ_{(\ell)}[v_1,v_2,\ldots]
\]
where $v_i = x_{\ell^i-1}$.
Convergence of the slice spectral sequences for $\MGL[1/p]$ and $\BP$
(for $\ell\ne p$)
is a special case of \cite[Theorem 8.12]{Hoyois}.

\begin{defn} \label{defn:SliceDiagonal}
A suspension $\Sigma^{m+n\alpha}$ is \emph{diagonal} if $m=n$.  A
spectrum $\EEE$ is \emph{slice diagonal} at $\ell$ if its slices are
wedges of diagonal suspensions of $\MZ_\ell$ (in which case $s_q\EEE$
is a potentially empty wedge of $q(1+\alpha)$-suspensions of $\MZ_\ell$).
\end{defn}

\begin{rmk} \label{rmk:SliceDiagonal}
The $\ell$-complete algebraic cobordism spectrum, the
$\ell$-complete motivic
Brown-Peterson spectrum $\BP$, and the $\ell$-complete truncated
Brown-Peterson spectra $\BPn{n}$ are all slice diagonal at $\ell$.
\end{rmk}

In order to manipulate the slice spectral sequence of slice diagonal
spectra over low-dimensional fields, we will need the following lemma
which records the known vanishing range for $\pi_\star\MZ_\ell$.

\begin{lemma}\label{lemma:vanish}
If $\cd_\ell(k)$ is the $\ell$-torsion cohomological dimension of
$k$, then $\pi_{m+n\alpha}\MZ_\ell = 0$ whenever $m<0$, $m>-n$, or
$m<-n-\cd_\ell(k)$.
\end{lemma}
\begin{proof}
The vanishing range for $\pi_\star \MZ_\ell$ follows from the
Bloch-Kato conjecture along with the definition of cohomological
dimension.  In particular, we have
\[
  \pi_{m+n\alpha} \MZ_\ell \cong
  \begin{cases}
    H^{-m-n}_{\text{\'{e}t}}(k;\ZZ_\ell(-n)) &\text{if }m\ge 0\\
    0&\text{otherwise}.
  \end{cases}
\]
\end{proof}

\begin{thm} \label{thm:slice}
Suppose $k$ is $\ell$-low-dimensional and $\EEE $ is slice diagonal at $\ell$.  
Then the slice spectral sequence for $\EEE$ collapses at the $E_1$-page.
\end{thm}
\begin{proof}
By Lemma \ref{lemma:vanish}, we know that
$\pi_{m+n\alpha}s_q\EEE = 0$ whenever $0 > 2q-(m+n)$ or $2q-(m+n) >
2$.  Now whenever $0\le 2q-(m+n)$ (so that
the source of a differential $d_r$ is possibly nonzero) we have
\[
  2(q+r)-(m-1+n) = (2q-(m+n)) + 2r+1 \ge 3,
\]
whence the target vanishes and $d_r = 0$.
\end{proof}

\begin{cor} \label{cor:MGL}
The coefficients of the $\ell$-complete algebraic cobordism spectrum over an
$\ell$-low-dimensional field are
\[
  \pi_\star \MGL\comp{\ell} = (\pi_\star \MZ_\ell)[x_1,x_2,\ldots]
\]
where $|x_i| = i(1+\alpha)$.
\end{cor}

\begin{cor} \label{cor:BP}
The coefficients of the $\ell$-complete motivic Brown-Peterson spectrum over
an $\ell$-low-dimensional field are
\[
  \pi_\star \BP\comp{\ell} = (\pi_\star \MZ_\ell)[v_1,v_2,\ldots]
\]
where $|v_i| = (\ell^i-1)(1+\alpha)$.
\end{cor}

\begin{rmk}
Under the above hypotheses, similar result holds for the truncated
motivic Brown-Peterson spectra
\[
  \BPn{n} = \BP/(v_{n+1},v_{n+2},\ldots),
\]
namely that
\[
  \pi_\star \BPn{n}\comp{\ell} = (\pi_\star \MZ_\ell)[v_1,\ldots,v_n].
\]
\end{rmk}

\begin{rmk}
There are many cases in which the slice spectral sequences for
$\MGL\comp{\ell}$ and related motivic spectra do not collapse.  The
authors determined $\pi_\star\BPn{n}\comp{2}$ and $\pi_\star\BP\comp{2}$ over $\QQ$ via the
motivic Adams spectral sequence in \cite{BPQ},  and Mike Hill covered the case of
base field $\RR$ in \cite{HillExt}.  Over both $\QQ$ and $\RR$, the slice spectral
sequence for $\BPn{n}\comp{2}$ does not collapse for $n\ge 1$, and the
slice spectral sequence for $\BP\comp{2}$ in fact has infinitely many
pages with nontrivial differentials.
\end{rmk}

\section{The motivic Adams-Novikov spectral sequence and its $E_2$-term} 
\label{sec:E2}

Here we recall basic facts about the motivic Adams-Novikov spectral sequence (MANSS) and then compute its $E_2$-term, 
the cohomology of the motivic Brown-Peterson Hopf algebroid, 
when $k$ is a low-dimensional field.

\subsection{The MANSS}

For each motivic spectrum $X$ and motivic homotopy ring spectrum $\EEE$ there is an associated $\EEE$-motivic Adams spectral sequence 
($\EEE$-MASS) for $X$ constructed in the following manner (completely analogous to the construction in, 
e.g., 
\cite[Lemma 2.2.9]{Ravenelbook}).

Set $X_0 = X$, let $K_s = \EEE\smsh X_s$, and let $X_{s+1}$ be the
fiber of $X_s\to K_s$, the map induced by the unit $\one\to \EEE$.
We get the $\EEE$-Adams resolution for $X$,

\begin{equation}
\label{adamstower}
\xymatrix{
\cdots\ar[r] & X_{s+1} \ar[r]\ar[d] & X_{s} \ar[r]\ar[d] & X_{s-1}
\ar[r]\ar[d] & \cdots \ar[r] & X_0=X\ar[d]\\
& K_{s+1} & K_s &K_{s-1} &&K_0 = \EEE\smsh X
}.
\end{equation}

Applying the bigraded homotopy group functor to the Adams resolution yields an exact couple and an associated spectral sequence in a standard way.
Our primary example is the $\BP$-MASS, 
where $\BP$ is the $\ell$-completion of the $\ell$-local motivic Brown-Peterson spectrum.  
For legibility we abuse notation by omitting a symbol for the $\ell$-completion.  
There is, of course, 
another $\BP$-MASS in which $\BP$ is not $\ell$-completed, 
but we will not use it here.  
A good reason for doing this is that $\pi_\star \MZ_\ell$ can be
computed in terms of \'{e}tale (or Galois) cohomology of the base field, 
while $\pi_\star \MZ_{(\ell)}$ is much more difficult.  
Translating our abuse of notation into an abuse of terminology, 
we will call the $\BP$-MASS (based on $\ell$-complete $\BP$), 
the ($\ell$-\emph{primary}) \emph{motivic Adams-Novikov spectral sequence} (MANSS).

\begin{thm} \label{thm:MANSS}
The MANSS over $k$ is a tri-graded multiplicative spectral sequence
with $E_2$-term
\[
  E_2^{s,m+n\alpha} = \Ext_{\BP_\star\BP}^{s,m+n\alpha}(\pi_\star \BP, \pi_\star\BP)
\]
and differentials of the form
\[
  d_r:E_r^{s,m+n\alpha}\to E_r^{s+r,m+r-1+n\alpha}.
\]
If $\ell=2$ and $k$ is a field of finite virtual cohomological dimension
(i.e., $\cd_2 k(i) < \infty$) with $p \ne 2$, or if $\ell\ne p$ and
$\cd_\ell k<\infty$ then the MANSS
converges to the stable motivic homotopy groups of $\one\comp{\ell}$,
\[
  E_2^{s,m+n\alpha}\implies \pi_{m-s+n\alpha} \one\comp{\ell}.
\]
\end{thm}
\begin{proof}
Assume for the moment that $k$ has characteristic $0$.  By
\cite[Theorem 1]{KO}, the motivic Adams spectral sequence 
converges to $\pi_\star \one\comp{\ell}$.  
By the same theorem, the $\M\FF_\ell$-nilpotent completion of $\BP$ is
equivalent to the $\ell$-completion of $\BP$.  Thus the MANSS is
realized by the cosimplicial spectrum \cite[(7.6)]{DI}.  In \cite[\S
7.3]{DI}, Dugger and Isaksen prove that the homotopy limit of this
cosimplicial objects is exactly the $\M\FF_\ell$-nilpotent completion
of $\one$.  Invoking \cite[Theorem 1]{KO} again, we see that the MANSS
convergences to $\one\comp{\ell}$.  Multiplicative properties are formal and follow by the
same argument as the proof of \cite[Theorem 2.3.3]{Ravenelbook}.

Now assume that $k$ has postivie characteristic $p$.  The above argument
goes through as long \cite[Theorem 1]{KO} holds in positive
characteristic with $\ell\ne p$.  The argument in \cite{KO} is in fact
independent of characteristic as long as the dual motivic Steenrod
algebra $\pi_\star \M\FF_\ell\smsh \M\FF_\ell$ has the same structure
as its characteristic $0$ counterpart.\footnote{The paper \cite{KO}
  also appeals to the Bloch-Kato conjecture, but this again has the
  same form since we avoid the $p$-part.}  This is precisely the main
theorem of \cite{HoyKO} with $\ell\ne p$, so we are done.
\end{proof}

\subsection{The cohomology of the motivic Brown-Peterson Hopf algebroid}

We now study the $E_2$-term of the motivic Adams-Novikov spectral
sequence.  For the moment, let $k$ be an arbitrary base field with
$p\ne \ell$.  We have
\[
  E_2 = \Ext_{\BP_\star\BP}(\pi_\star \BP,\pi_\star
\BP),
\]
which in turn can be computed as the cohomology of the cobar
complex for the Hopf algebroid $(\pi_\star\BP,\BP_\star\BP)$.  (See
\cite[\S A1]{Ravenelbook} for details on the homological algebra of Hopf
algebroids.)  

Recall that by \cite[Proposition 9.1]{NSO2}, we have
\begin{equation} \label{eqn:coops}
  \BP_\star\BP \cong \BP^\top_*\BP^\top \otimes_{\pi_* \BP^\top}
  \pi_\star \BP
\end{equation}
where $\BP^\top$ is the topological Brown-Peterson spectrum.  The
graded rings $\pi_*\BP^\top$ and $\BP^\top_*\BP^\top$ refer to the
usual topological coefficients and cooperations put in motivic degrees
$\frac{*}{2}(1+\alpha)$.  This implies that the cobar complex
$C^*(\BP)$ for
\[
  (\pi_\star\BP,\BP_\star\BP)
\]
decomposes as
\begin{equation} \label{eqn:cobar}
  C^*(\BP) \cong C^*(\BP^\top) \otimes_{\pi_\star \BP^\top} \pi_\star \BP.
\end{equation}
Here $C^*(\BP^\top)$ refers to the usual cobar complex for
$(\pi_*\BP^\top,\BP^\top_*\BP^\top)$ with internal degrees shifted to
$\frac{*}{2}(1+\alpha)$.  Via (\ref{eqn:cobar}), we get the following
spectral sequence.

\begin{prop} \label{prop:UCSS}
Let
\[
  E_2^{s,m+n\alpha} =
\Ext_{\BP_\star\BP}^{s,m+n\alpha}(\pi_\star\BP,\pi_\star\BP)
\]
denote
the $E_2$-term of the MANSS and let
\[
  {}^\top E_2^{s,t} =
  \Ext_{\BP^\top_*\BP^\top}(\pi_*\BP^\top,\pi_*\BP^\top)
\]
denote the
$E_2$-term of the topological ANSS
grade-shifted so that $\Ext^{s,t}$ appears in motivic tri-degree
$(s,\frac{t}{2}(1+\alpha))$.  Consider ${}^\top E_2^{*,*}$ as a
$\pi_*\BP^\top$-module via the natural action of $\pi_*\BP^\top$ on
$C^*(\BP^\top)$.  There is a quint-graded universal
coefficient spectral sequence of homological type
\begin{equation} \label{eqn:UCSS}
  E^2_{r,s,t,m+n\alpha} = \Tor^{\pi_\star\BP^\top}_r({}^\top
  E_2^{s-r,t},\pi_{m+n\alpha}\BP) \implies E_2^{s,(m+\frac{t}{2})+(n+\frac{t}{2})\alpha}
\end{equation}
converging to the $E_2$-term of the MANSS.
\end{prop}
\begin{proof}
This is a direct consequence of the homology K\"unneth spectral
sequence \cite[Theorem 10.90]{Rotman}.
\end{proof}

\begin{rmk}
Given the intractability of the $\pi_*\BP^\top$-module structure on $C^*(\BP^\top)$, 
(\ref{eqn:UCSS}) is difficult to work with in cases when (\ref{eqn:cobar}) does not
simplify.  When $k$ is an $\ell$-low-dimensional field we will 
observe precisely such a simplification.  For arbitrary $k$, it would
be beneficial to have an alternate
``topological-to-motivic'' spectral sequence computing the motivic
Adams-Novikov $E_2$, perhaps one that takes advantage of extra
structure available in the stable motivic homotopy category.
\end{rmk}

We now specialize to $k$ an $\ell$-low-dimensional field.  By Corollary
\ref{cor:BP}, in this case we have
\[
  \pi_\star\BP = (\pi_\star \MZ_\ell)[v_1,v_2,\ldots] = \pi_*\BP^\top
  \otimes_{\ZZ_\ell} \pi_\star \MZ_\ell.
\]
As such, we can rewrite (\ref{eqn:cobar}) as
\begin{equation} \label{eqn:ldcobar}
  C^*(\BP) \cong C^*(\BP^\top)\otimes_{\ZZ_\ell} \pi_\star \MZ_\ell.
\end{equation}
We can thus apply the universal coefficient spectral sequence (now
better known as the universal coefficient theorem) over $\ZZ_\ell$ to
arrive at the following description of the MANSS $E_2$-term.

\begin{thm} \label{thm:UCT}
If $k$ is an $\ell$-low-dimensional field, then the $E_2$-term of the MANSS
sits in a (non-canonically) split short exact sequence
\begin{multline} \label{eqn:UCT}
  0\to {}^\top E_2^{s,t} \otimes \pi_{m+n\alpha}\MZ_\ell \to
  E_2^{s,(m+\frac{t}{2})+(n+\frac{t}{2})\alpha}\\ \to
  \Tor^{\ZZ_\ell}_1({}^\top E_2^{s-1,t},\pi_{m+n\alpha}\MZ_\ell) \to 0.
\end{multline}
\end{thm}

In what follows, we will frequently abbreviate $\Tor_1^{\ZZ_\ell}$ to $\Tor$.

\section{The first $\ell$-complete stable motivic stem, $\ell>2$} 
\label{sec:largePrimePi1}

Let us now refer to the groups $\pi_{m+n\alpha} \one\comp{\ell}$ as the
$\ell$-complete stable motivic stems. For $k$ $\ell$-low-dimensional, we will use
the explicit decomposition of the MANSS $E_2$-term in
Theorem \ref{thm:UCT} to perform computations in the $\ell$-complete
stable motivic stems.

Recall that the MANSS based on $\ell$-complete $\BP$ converges as
\[
  E_2^{s,m+n\alpha} \implies \pi_{m-s+n\alpha}\one\comp{\ell}.
\]
Henceforth we will follow the good practice of depicting and thinking
about the MANSS in ``Adams grading.''  We will fix a
weight $n\alpha$ and think of the MANSS as a bi-graded spectral
sequence in $s$ and $m$.  We will then plot (either graphically or
mentally) the integers $m-s$ along the horizontal axis and $s$ along
the vertical axis (with $n$ suppressed).  In such a grading scheme,
the differential $d_r$ originating in coordinates $(m-s,s)$ has target
with coordinates 
$(m-s-1,s+r)$, one unit to the left and $r$ units up.  We will refer
to the portion of the MANSS in $E_2^{s,m+n\alpha}$ with $m-s+n\alpha$
fixed and $s$ varying as the $(m-s+n\alpha)$-column of the MANSS.

Our strategy for computing $\pi_1 S\comp{\ell}$ is simple:  compute the $0$-, $1$-,
and $2$-columns of the $\ell$-primary MANSS and then make various arguments regarding
potential differentials.   The computations break into three pieces:
$\ell>3$ (very easy), $\ell=3$
(easy), and $\ell=2$ (hard).  In this section, we address $\ell>3$ and $\ell=3$
as a warmup for $\ell=2$ in the next section.

\subsection{Primes greater than $3$}

Let us pose and answer the following question:  given Theorem
\ref{thm:UCT} and the structure of $\pi_\star \M\ZZ_\ell$, what appears
in the $1$-column of the $\ell$-MANSS?  First off, we have contributions of
the form ${}^\top E_2^{s,t}\otimes \pi_{m+n\alpha} \M\ZZ_\ell$ whenever
\[
  \frac{t}{2}(1+\alpha)-s + m+n\alpha = 1.
\]
In this case, $t-s+m+n = 1$, and since $k$ is $\ell$-low-dimensional
we can use \ref{lemma:vanish} to 
see that $1\le t-s\le 3$.  (This uses the fact that the $0$-column of
${}^{\top}E_2$ is concentrated in homological degree $s=0$.)  For any prime $\ell$, the first entry after
$\ZZ_\ell = {}^\top E_2^{0,0}$ in the topological ANSS is
$\ZZ/\ell\{\alpha_1\}$ with homological degree $1$ and total dimension
$2\ell-3$; there are no other nontrivial terms with the same total
dimension in the interval $[1,2\ell-3]$.  In particular, if $\ell>3$ there are no
contributions to the $1$-column of the form ${}^\top E_2^{s,t}\otimes
\pi_{m+n\alpha} \M\ZZ_\ell$.

The other contributions to the $1$-column are of the form
$\Tor({}^\top E_2^{s,t},\pi_{m+n\alpha} \M\ZZ_\ell)$ with
\[
  \frac{t}{2}(1+\alpha)-s+m+n\alpha = 0.
\]
By the same logic, $1\le t-s\le 2$, and we see that there are no
$\Tor$ contributions to the $1$-column for $\ell\ge 3$.

In summary, 
the $1$-column of the $\ell$-MANSS vanishes when $\ell>3$, 
and we deduce the following.

\begin{thm} \label{thm:p>3pi1}
If $k$ is $\ell$-low-dimensional and $\ell>3$, then
\[
 \pi_{1} \one\comp{\ell} = 0.
\]
\qed
\end{thm}

\subsection{The prime $3$}

By the preceding discussion, we see that when $\ell=3$ there is a unique
contribution to the $1$-column of the $3$-MANSS of the form
\[
  {}^\top E_2^{1,4}\otimes \pi_{-2\alpha}\M\ZZ_3
\]
since the 3-column of ${}^{\top}E_2$ is concentrated in filtration
$s=1$.  Of course, 
the first factor is isomorphic to $\ZZ/3\{\alpha_1\}$ and the second
is $K^M_2(k)\comp{3}$, the 3-completion of $K^M_2(k)$.   (See, e.g., \cite[p.501]{Zahler}
for a picture of the $E_2$-page of the topological ANSS at the prime
3.)  
By similar bookkeeping, the $0$-column is concentrated in homological degree $0$, 
so our $1$-column entry is an infinite cycle.  
Additionally, 
 $\ZZ/3\{\alpha_1\}\otimes K^M_2(k)\comp{3}$ has homological degree $1$, 
and hence survives the spectral sequence.

\begin{thm} \label{thm:3pi1}
If $k$ is $3$-low-dimensional, then
\[
\pi_1 \one\comp{3} \cong K^M_2(k)/3.
\]
\end{thm}

\section{The first $2$-complete stable motivic stem} 
\label{sec:2pi1}

We now execute the same strategy when $\ell=2$ and $k$ is a $2$-low-dimensional field.  
It turns out that all the classes in the $1$-column of the $2$-MANSS are infinite cycles, 
but there exists a nontrivial entering $d_3$-differential.  
There are also several nontrivial extensions.  
Unlike when $\ell\ge 3$, 
there is real work to do when $\ell=2$ and our arguments are not strictly combinatorial.
In this section, we shall refer to the $2$-MANSS as simply the MANSS.

We begin by reviewing some known facts about ${}^\top E_2^{*,*}$ and
$\pi_\star\MZ_2$.  By Lemma \ref{lemma:vanish} and the
low-dimensionality of $k$, we know 
that $\pi_{m+n\alpha}\MZ_2 = H^{-(m+n)}(k;\ZZ_2(-n)) = 0$ for
$m<0$, $n>-m$, or $-m-2>n$.

\begin{figure}[h!]
\centering
\includegraphics[width=4.5in]{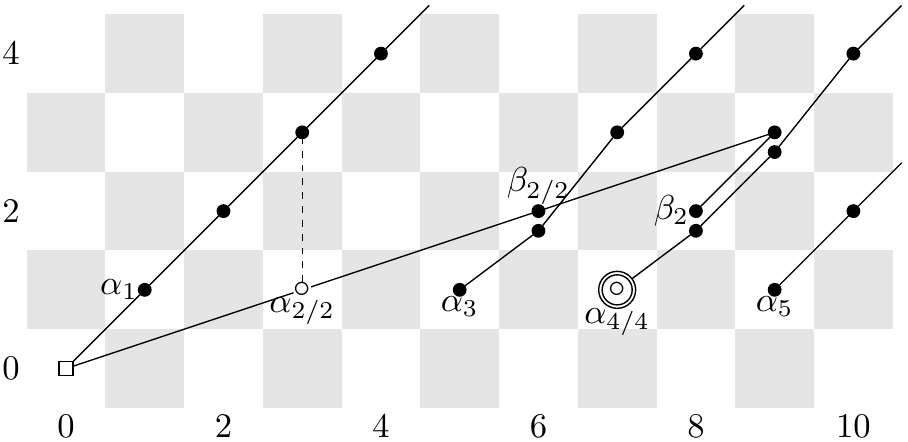}
\caption{The $E_2$ page of the topological ANSS} \label{fig:ANSS}
\end{figure}

Figure \ref{fig:ANSS} is a graphical representation of ${}^\top E_2^{s,t}$
in which the horizontal axis represents $t-s$ and the vertical axis
represents $s$.  (This is the standard, ``Adams graded'' way of
drawing the picture in which the $(t-s)$-column contains all the
potential $\pi_{t-s}$ contributions.)  This picture is well-known to
algebraic topologists and can be found in, e.g., \cite[p.500]{Zahler};
see \cite{MRW} or \cite{Ravenelbook} for the naming conventions and many
more details.
Here $\square$ represents a copy of $\ZZ_2$, $\bullet$ represents
$\ZZ/2$, $\circ$ represents $\ZZ/4$, and $\circ$ with two circles
around it represents $\ZZ/16$.

Theorem \ref{thm:UCT} provides a method for combining this picture
with $\pi_\star\MZ_2$ in order to produce the $E_2$ page of the
MANSS.  Figure \ref{fig:MANSS} presents this data in the following form:  each
ANSS term $A$ produces $A\otimes \pi_\star \MZ_2$ and
$\Tor(A,\pi_\star \MZ_2)$ terms in the MANSS $E_2$.  If we draw our
chart for $E_2^{s,m+n\alpha}$ so that $m-s+n$ is on
the horizontal axis and $s$ is on the vertical axis (and keep in mind
our vanishing results on $\pi_\star\MZ_2$ when $k$ is
2-low-dimensional), then the $A\otimes \pi_\star \MZ_2$ terms appear in the
original position of $A$ (in the topological ANSS $E_2$ chart) and the two positions
to the left of this location.  Similarly, the
$\Tor(A,\pi_\star\MZ_2)$ contribution appears one down and one down
and left of the original location of $A$.  So if $A$ is represented by
a graphical primitive $\bullet$ in the ANSS $E_2$ chart, we represent
its contributions to the MANSS $E_2$ chart as the following
color-coded symbols.

\begin{center}
\includegraphics[width=1in]{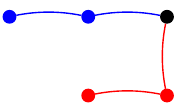}
\end{center}

\noindent Here the blue terms come from tensoring with $\pi_\star\MZ_2$ and the
red terms come from $\Tor(-,\pi_\star\MZ_2)$.  (This graphic and
Figure \ref{fig:MANSS} are best viewed in color.)  Note that there is
no $\Tor$ term contributed by $\pi_{n-n\alpha}\MZ_2$ because these
groups are free.

It is important to keep in mind that the weight $n$ is
suppressed in our picture.  Weights may be recovered by recalling that an ANSS term
in bidegree $s,t$ has weight $\frac{1}{2}t$ and then applying the
universal coefficient theorem.  Single graphical primitives represent
contributions to infinitely many different weights.

\begin{figure}[h!]
\centering
\includegraphics[width=5.5in]{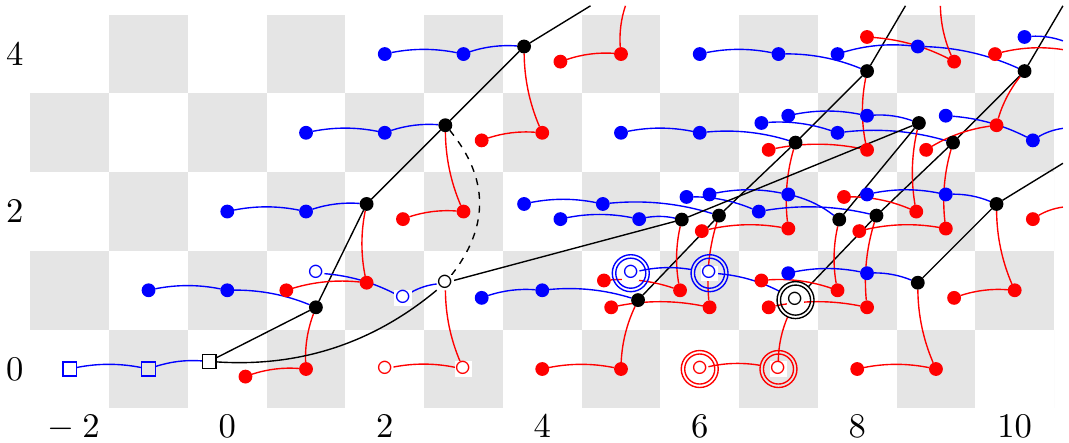}
\caption{The $E_2$ page of the MANSS} \label{fig:MANSS}
\end{figure}

\begin{prop} \label{prop:E2}
After choosing splittings in Theorem \ref{thm:UCT}, the $0$-, $1$-,
and $2$-columns of the $E_2$-page of the MANSS over a 
$2$-low-dimensional field take the form depicted in Table \ref{table:E2}.
\end{prop}
\begin{proof}
This is a rote calculation given Theorem \ref{thm:UCT}, the form of
${}^\top E_2^{*,*}$, and Lemma \ref{lemma:vanish} on the vanishing
range of $\pi_\star \MZ_2$.
\end{proof}

\renewcommand{\arraystretch}{1.5}
\begin{table}
\centering
\begin{tabular}[h]{|c|c|c|}
\vdots&\vdots&\vdots\\\hline

0&0&0\\\hline

0&0&$\ZZ/2\{\alpha_1^4\}\otimes \pi_{2-4\alpha}\MZ_2$\\\hline

0&$\ZZ/2\{\alpha_1^3\}\otimes \pi_{1-3\alpha}\MZ_2$&
$\ZZ/2\{\alpha_1^3\}\otimes \pi_{2-3\alpha}\MZ_2$\\\hline

$\ZZ/2\{\alpha_1^2\}\otimes
\pi_{-2\alpha}\MZ_2$
&$\ZZ/2\{\alpha_1^2\}\otimes
\pi_{1-2\alpha}\MZ_2$
&$\begin{array}{c}\ZZ/2\{\alpha_1^2\}\otimes
\pi_{2-2\alpha}\MZ_2\\ \oplus \\
\Tor(\ZZ/2\{\alpha_1^3\},\pi_{1-3\alpha}\MZ_2) \end{array}$\\\hline

$\ZZ/2\{\alpha_1\}\otimes
\pi_{-\alpha}\MZ_2$
&$\begin{array}{c}\ZZ/2\{\alpha_1\}\otimes
\pi_{1-\alpha}\MZ_2\\ \oplus\\
\Tor(\ZZ/2\{\alpha_1^2\},\pi_{-2\alpha}\MZ_2)\\ \oplus \\
\ZZ/4\{\alpha_{2/2}\}\otimes
\pi_{-2\alpha}\MZ_2 \end{array}$
&$\begin{array}{c} \Tor(\ZZ/2\{\alpha_1^2\},\pi_{1-2\alpha}\MZ_2) \\
  \oplus \\ \ZZ/4\{\alpha_{2/2}\}\otimes
  \pi_{1-2\alpha}\MZ_2\end{array}$\\\hline

$\ZZ_2\{1\}\otimes\pi_0\MZ_2$&$\Tor(\ZZ/2\{\alpha_1\},\pi_{-\alpha}\MZ_2)$&
$\Tor(\ZZ/4\{\alpha_{2/2}\},
\pi_{-2\alpha}\MZ_2)$\\\hline
0&1&2
\end{tabular}
\caption{The 0-, 1-, and 2-columns of the MANSS over a low-dimensional
  field.}
\label{table:E2}
\end{table}
\renewcommand{\arraystretch}{1.0}

We will need the following lemma about the simplicial Hopf map.

\begin{lemma} \label{lemma:topmot}
The simplicial Hopf map $\eta_s$ is nontrivial of order 2 in $\pi_1\one\comp{2}$
over any scheme.
\end{lemma}
\begin{proof}
It suffices to check the result over $\Spec \ZZ$. It is observed in
\cite{QuillenMilnorLetter} that (working in topological spectra) $\eta_s$
represents the class of $-1$ in $K_1(\ZZ)
\cong \pi_1\KGL \cong \{\pm 1\}$.  Consider the Quillen pair
$(c^*,c_*)$ in which $c^*$ is induced by the constant presheaf
functor.  Since (reinterpreting the construction of $\KGL$ in
\cite{Hornbostel}) $c_*\KGL \simeq K(k)$ (where $K(k)$ is the
topological spectrum representing algebraic $K$-theory of $k$) and the adjoint of the unit
map $\one\to \KGL$ is the unit map for $K(k)$, we find
that the image of $c^*\eta_s$ under the Hurewicz map for $\KGL$ also
corresponds to the generator $\eta_s$ of $\pi_1\KGL$.
\end{proof}

\begin{prop} \label{prop:perm}
Over a $2$-low-dimensional field, every differential $d_r$, $r\ge 2$,
exiting the 1-column of the MANSS is trivial.
\end{prop}
\begin{proof}
By Proposition \ref{prop:E2}, specifically the form of the 0- and
1-columns, there is only one potentially nontrivial differential,
\[
  d_2: \Tor(\ZZ/2\{\alpha_1\},\pi_{-\alpha}\MZ_2)\to
  \ZZ/2\{\alpha_1^2\}\otimes \pi_{-2\alpha}\MZ_2
\]
from the $(1,0)$-coordinate to the $(0,2)$-coordinate.  We argue that
the source of this differential is either $0$ or
represents $\eta_s$ (modulo higher filtration elements).  By Lemma
\ref{lemma:topmot}, this enough to conclude that the above $d_2$ is
trivial.

Recall that $\tau$ is the generator of $\pi_{1-\alpha}\M\FF_2$.  It
follows from \cite[p.968]{MorelMASS} that $h_1 \tau$ is
an infinite cycle in the motivic Adams spectral sequence.  As such, we
have a map $\pi_{1-\alpha}\M\FF_2\to \pi_1\one\comp{2}$ given by
multiplication by $\eta$.  By Lemma \ref{lemma:topmot}, this map is
injective since $h_1 \tau$ represents $\eta_s = h_1\tau+h_2 \rho^2$ over any field in which
$\rho^2=0$.

The standard cofiber sequence
$\M\ZZ_2\xrightarrow{2} \M\ZZ_2 \to \M\FF_2$ induce a short exact
sequence
\[
  0\to \ZZ/2\otimes \pi_{1-\alpha}\M\ZZ_2\to \pi_{1-\alpha}\M\FF_2\to
  \Tor(\ZZ/2,\pi_{-\alpha}\M\ZZ_2)\to 0.
\]
(Note that, depending on the base field, either the first term is
$\ZZ/2$ and the last is $0$, or vice versa.) In order for the motivic
Adams-Novikov spectral sequence to detect
the image of $h_1 \tau$ in $\pi_1\one\comp{2}$, the above must be an
extension in the motivic ANSS.  Since $h_2$ corresponds to $\nu$ which
in turn corresponds to $\alpha_{2/2}$, whenever
$\Tor(\ZZ/2\{\alpha\},\pi_{-\alpha}\M\ZZ_2)$ is nonzero it must detect
$\eta_s$ mod higher filtration, hence the $d_2$ in question is trivial.
\end{proof}

\begin{prop} \label{prop:triv_d2}
Over a $2$-low-dimensional field, the MANSS differential
\[
  d_2:\Tor(\ZZ/4\{\alpha_{2/2}\},\pi_{-2\alpha}\MZ_2)\to
  \ZZ/2\{\alpha_1^2\}\otimes \pi_{1-2\alpha}\MZ_2
\]
from the $(2,0)$- to $(1,2)$-coordinate is trivial.
\end{prop}
\begin{proof}
We invoke the multiplicative structure of the MANSS to prove this
result.  Let $1$ denote the class of the identity map and note that the multiplication
\[
  (\ZZ_2\{1\}\otimes \pi_{-\alpha}\MZ_2)\otimes
  \Tor(\ZZ/4\{\alpha_{2/2}\},\pi_{-\alpha}\MZ_2) \to \Tor(\ZZ/4\{\alpha_{2/2}\},\pi_{-2\alpha}\MZ_2)
\]
is a surjection (this can be checked on cobar representatives), 
and the classes in $\ZZ_2\{1\}\otimes \pi_{-\alpha}\MZ_2$ are permanent.  
Thus to determine the differential of the proposition we can analyze $d_2$ on
$\Tor(\ZZ/4\{\alpha_{2/2}\},\pi_{-\alpha}\MZ_2)$.  By Theorem
\ref{thm:UCT}, the target of this $d_2$ is
$(\ZZ/2\{\alpha_1^2\}\otimes \pi_{1-\alpha}\MZ_2)\oplus
\Tor(\ZZ/2\{\alpha_1^3\},\pi_{-2\alpha}\MZ_2)$.  Multiplying by
$\ZZ_2\{1\}\otimes \pi_{-\alpha}\MZ_2$, we see that the proposition
follows as long as $d_2$ on
\[
  \Tor(\ZZ/4\{\alpha_{2/2}\},\pi_{-\alpha}\MZ_2)
\]
composed with
projection onto $\ZZ/2\{\alpha_1^2\}\otimes \pi_{1-\alpha}\MZ_2$ is
trivial.  (Since $k$ is $2$-low-dimensional, the multiplication on 
$(\ZZ_2\{1\}\otimes\pi_{-\alpha}\MZ_2)\otimes\Tor(\ZZ/2\{\alpha_1^3\},\pi_{-2\alpha}\MZ_2)$ is trivial.)

We now make the following claim about the coefficients of $\MZ_2$:
\begin{equation} \label{eqn:claim}
\text{If }
\pi_{1-\alpha}\MZ_2\ne 0\text{, then }\pi_{-\alpha}\MZ_2\cong
\varprojlim_r k^\times/(k^\times)^{2^r}\text{ is
2-torsion-free.}
\end{equation}
This claim follows from the isomorphism
\[
  \pi_{1-\alpha}\M\ZZ_2 \cong \varprojlim_r H^0_{\text{\'{e}t}}(\Spec k;
  \mu_{2^r})
\]
and direct computation of the limit on the right-hand side.

By (\ref{eqn:claim}), we now see that whenever
$\Tor(\ZZ/4\{\alpha_{2/2}\},\pi_{-\alpha}\MZ_2)$ is nontrivial, the
group $\ZZ/2\{\alpha_1^2\}\otimes \pi_{1-\alpha}\MZ_2$ vanishes, proving that
the $d_2$-differential in question is trivial.
\end{proof}

\begin{prop} \label{prop:4nu}

Over a field $k$ with $p \ne 2$ and $\cd_2 k(i) < \infty$, we have the relation 
$4\nu = \eta^2\eta_s$ in $\pi_{1+2\alpha}\one\comp{2}$.
\end{prop}

\begin{rmk} \label{rmk:4nu}
Proposition \ref{prop:4nu} is a $2$-adic motivic version of the relation $12\nu = \eta^3$ in the topological $\pi_3S$.  
This has ramifications in the first motivic stable stem:  
First,  
identify $\ZZ/2\{\alpha_1^3\}\otimes \pi_{1-2\alpha}\MZ_2$ appearing
in the weight one MANSS with 
$\pi_{1-2\alpha}\M\FF_2\{\alpha_1^3\} =
(K^M_1(k)/2)\{\tau\alpha_1^3\}$.  (See \S\ref{subsec:wt1} for more
information on this spectral sequence.)  
Then we have that
\begin{equation} 
\label{eqn:4nu}
4[u,v]\alpha_{2/2} 
= 
[u,v]\tau\alpha_1^3
\end{equation}
in the $1$-column of the MANSS.
\end{rmk}

\begin{proof}
Dugger and Isaksen \cite{DI} observe that $h_0^2h_2 = \tau h_1^3$ 
in the $E_2$-page of the motivic Adams spectral sequence.  (Their 
computations are done over the base field $\CC$, but this particular 
relation may be checked in the cobar complex.  The computation of
$\M\FF_2^\star\M\FF_2$ in \cite{HoyKO} implies 
that the argument works in characteristic $p>2$.)  Over 
the prime field of $k$, it is known that $h_0$ represents $2+\rho\eta$, 
$h_1$ represents $\eta$, $h_1\tau+h_2\rho^2$ represents $\eta_s$ (the simplicial 
Hopf element), and $h_2$ represents $\nu$.  Dugger and Isaksen
\cite{DDDI} also have a 
geometric argument showing that $\eta \nu = 0$, 
whence their relation in the MASS becomes $4\nu =
\eta^2(\eta_s-\nu\rho^2) = \eta^2\eta_s$.
\end{proof}

\begin{prop} \label{prop:d2triv}
The MANSS differential
\[
  d_2: \begin{array}{c} \Tor(\ZZ/2\{\alpha_1^2\},\pi_{1-2\alpha}\MZ_2)\\
  \oplus \\ \ZZ/2\{\alpha_{2/2}\}\otimes \pi_{1-2\alpha}\MZ_2 \end{array} \to
\ZZ/2\{\alpha_1^3\}\otimes \pi_{1-3\alpha}\MZ_2
\]
from the $(2,1)$-coordinate to the $(1,3)$-coordinate is trivial.  In
fact, the source of this differential is permanent in
the MANSS.
\end{prop}
\begin{proof}
We show that $d_2$ is trivial on each summand.  
Note that the $\Tor$-term is either $0$ or $\ZZ/2$ depending on whether $\mu_{2^\infty}(k)$ is infinite or finite, 
respectively.  
In the latter case, by an argument similar to Proposition
\ref{prop:perm}'s, 
this group is $\ZZ/2$ and represents $\eta_s^2$ mod higher
filtration.  Since $\eta_s^2\ne 0\in \pi_2 \KGL$ over $\ZZ$, we can
conclude that this class is permanent.

Now note that $\pi_{1-2\alpha}\MZ_2$ is permanent in the MANSS by
dimensional accounting.  Meanwhile, $\alpha_{2/2}$ is permanent
because its potential targets (via $d_2$ and $d_3$) are $K^M_1(k)/2\{\alpha_1^3\}$
and $K^M_2(k)/2\{\alpha_1^3\}$.  The potential targets survive via
Morel's calculation of the $0$-line $\pi_{*\alpha}\one$.  By the Leibniz rule, we find that the
second summand is permanent.
\end{proof}

In order to state the following proposition, we need to know some
facts about $\Tor(\ZZ/4\{\alpha_{2/2}\},\pi_{-2\alpha}\M\ZZ_2)$ in the
MANSS.  For $a,b\in k^\times$, let $s(a,b)$ be the sum of Milnor symbols $[a,a+1]+[b,b^2+1]$.  By
results of Browkin \cite[Theorem 4.2]{Browkin},
\[
  \Tor(\ZZ/4,K^M_2(k)) = \{s(a,b) \mid a,b\in k^\times\} \cong \Tor(\ZZ/4,\pi_{-2\alpha}\M\ZZ_2).
\]
By the proof of the universal coefficient theorem, cobar
representatives of the elements in
$\Tor(\ZZ/4\{\alpha_{2/2}\},\pi_{-2\alpha}\M\ZZ_2)$ are of the form
$f\otimes s(a,b)$, where $f$ satisfies $d_1 f =
4\tilde{\alpha}_{2/2}$ and $\tilde{\alpha}_{2/2}$ is a cobar
representative of $\alpha_{2/2}$ in the topological ANSS.  A rote
calculation shows that we may take $f = v_1^2$.  We define
$\sigma_{a,b}\alpha_{2/2}$ to be the MANSS $E_2$ element
represented by $f\otimes s(a,b)$.

\begin{prop} \label{prop:d3}
The MANSS differential
\[
  d_3:\Tor(\ZZ/4\{\alpha_{2/2}\},\pi_{-2\alpha}\M\ZZ_2)\to
  \ZZ/2\{\alpha_1^3\}\otimes \pi_{1-3\alpha}\MZ_2
\]
takes the form
\[
  d_3(\sigma_{a,b}\alpha_{2/2}) = s(a,b)\tau\alpha_1^3.
\]
\end{prop}

\begin{proof}
The proof proceeds by analysis of the MANSS for the mod-$2$ Moore spectrum $\one/2$, 
i.e., 
the cofiber of $\one\xrightarrow{2}\one$.  The MANSS $E_2$ for
$\one/2$ over a low-dimensional field takes a particularly simple form.  
Via the slice spectral sequence and a simple change of base, 
we see that
\[
E_2(\one/2) 
= 
{}^\top E_2(\one/2)\otimes_{\FF_2}\pi_\star \M\FF_2.
\]
By \cite[Theorem 5.3.13(a)]{Ravenelbook}, 
this $E_2$-term contains 
\[
  \pi_\star\M\FF_2[v_1,h_0]
\]
as a direct summand where $v_1$ has Adams degree $1+\alpha$ and
homological degree $0$ and $h_0$ has Adams degree $\alpha$ and
homological degree $1$.  We know that the mod $2$ reduction map $\red$ induces a map 
of spectral sequences which takes $\alpha_1$ to $h_0$.  

Since $\eta^2\eta_s = 4\nu$, 
we know that $\red(\tau\alpha_1^3) = \tau h_0^3$ represents a trivial class in $\pi_\star \one/2$.  
It follows that $\tau h_0^3$ dies in the MANSS for $\one/2$.  
The only way this can happen is if
\[
d_3 v_1^2 
= 
\tau h_0^3
\]
whence
\[
d_3 [u,v]v_1^2 
=
 [u,v]\tau h_0^3
\]
for all $[u,v]\in K^M_2(k)/2$.  
(The symbols $[u,v]$ are permanent by Morel's computation of $\pi_{*\alpha}\one$ and the cofiber sequence defining $\one/2$.)

To determine the value of $d_3$ in the MANSS for $\one$ on $\sigma_{s(a,b)}\alpha_{2/2}$, 
we look at its value under the reduction map $\red$ to $\one/2$.
Since $\sigma_{a,b}\alpha_{2/2}$ is represented by
$s(a,b)v_1^2$ in the cobar complex we find that
\[\begin{aligned}
  \red d_3\sigma_{a,b}\alpha_{2/2} &= d_3 s(a,b)v_1^2\\
  &= s(a,b)\tau h_0^3.
\end{aligned}\]
This is only possible if
\[
  d_3\sigma_{a,b}\alpha_{2/2} \equiv s(a,b)\tau \alpha_1^3\pmod{2}.
\]
Since $2\alpha_1=0$, this implies that
\[
  d_3\sigma_{a,b}\alpha_{2/2} = s(a,b)\tau\alpha_1^3,
\]
which is the claimed differential.
\end{proof}

\begin{rmk}
The standard relations in Milnor-Witt $K$-theory imply that the class 
$\rho[u]\tau\alpha_1^3$ must die in the MANSS.  Indeed, since
$(2+\rho\eta)\eta = 0$, it follows that $\rho\eta^2\eta_s =
-2\eta\eta_s = 0$ since $2\eta_s=0$.  It is fascinating to note that the differential in
Proposition \ref{prop:d3} witnesses torsion on $\eta^2\eta_s$ that is
not a trivial consequence of its module structure over Milnor-Witt $K$-theory.
\end{rmk}

The preceding propositions account for all possible MANSS differentials which affect the $1$-column.  
It follows that the $1$-column of the $E_\infty$-page is exactly that of Table \ref{table:E2} with the exception that 
$\ZZ/2\{\alpha_1^3\}\otimes \pi_{1-3\alpha}\MZ_2$ gets replaced by
\[
  K^M_2(k)/(2,\Tor(\ZZ/4,K^M_2(k)))\{\tau\alpha_1^3\},
\] 
where $\eta^2\eta_s$ is represented by $\tau\alpha_1^3$.  
In order to determine $\pi_1\one\comp{2}$, 
it only remains to solve the extension problems.

\begin{prop} \label{prop:Kmod8}
There is a hidden extension joining $\ZZ/4\{\alpha_{2/2}\}\otimes
\pi_{-2\alpha}\MZ_2$ and
\[
  K^M_2(k)/(2,\Tor(\ZZ/4,K^M_2(k)))
\]
so that these
groups represent $(K^M_2(k)/8)\{\nu\}$
in $\pi_1\one\comp{2}$.
\end{prop}
\begin{proof}
In Proposition \ref{prop:4nu} we already observed that there is a
hidden extension
\[
  4\alpha_{2/2} = \tau \alpha_1^3.
\]
Now tensor the
short exact sequence
\[
  0 \to \ZZ/2\to \ZZ/8\to \ZZ/4\to 0
\]
with $K^M_2(k)$ and observe that there is an associated long exact
sequence
\[
  \cdots \to \Tor(\ZZ/4,K^M_2(k))\to K^M_2(k)/2 \to K^M_2(k)/8\to
  K^M_2(k)/4\to 0.
\]
It follows that
$K^M_2(k)/8$ sits in a short exact sequence
\[
  0\to K^M_2(k)/(2,\Tor(\ZZ/4,K^M_2(k))) \to K^M_2(k)/8\to
  K^M_2(k)/4\to 0.
\]
The relation $4\alpha_{2/2} = \tau \alpha_1^{3}$ implies that this is
an extension in the MANSS.
\end{proof}

We now record some further extensions in the 1-column of the MANSS.

\begin{prop} \label{prop:UCT}
There are extensions
\[
  0\to \ZZ/2\{\alpha_1\}\otimes \pi_{1-\alpha}\M\ZZ_2\to
\ZZ/2\{\eta_s\}\to \Tor(\ZZ/2\{\alpha_1\},\pi_{-\alpha}\M\ZZ_2)\to 0
\]
and
\[
  0\to \ZZ/2\{\alpha_1^2\}\otimes \pi_{1-2\alpha}\M\ZZ_2\to
  K^M_1(k)/2\{\eta\eta_s\}\to
  \Tor(\ZZ/2\{\alpha_1^2\},\pi_{-2\alpha}\M\ZZ_2)\to 0
\]
in the motivic Adams-Novikov spectral sequence.
\end{prop}
\begin{proof}
The first extension has already been observed in the proof of
Proposition \ref{prop:perm}.  For the second extension, first note
that such a short exact sequence of abelian groups exists by applying
$\pi_{1-2\alpha}$ to the cofiber sequence
$\M\ZZ_2\xrightarrow{2}\M\ZZ_2\to \M\FF_2$.  (We use the isomorphism
$\pi_{1-2\alpha}\M\FF_2\cong K^M_1(k)/2$ to make the identification.)
We see that this extension is realized by multiplying the first
extension by $\alpha_1$.
\end{proof}

The unit map $\one\to \KO$ for Hermitian $K$-theory over any field with $p\neq 2$ induces 
\[
\pi_1\one\comp{2} 
\xrightarrow{e} 
\pi_1\KO\comp{2} 
\cong 
K^M_1(k)/2\oplus \ZZ/2.
\]
Classically \cite{Bass}, the isomorphism $\pi_1\KO\cong K^M_1(k)/2\oplus
\ZZ/2$ is given by $(SN,\det)$ where $SN$ is the spinor norm and
$\det$ is the determinant composed with the isomorphism $\{\pm
1\}\cong \ZZ/2$.

\begin{lemma} \label{lemma:surj}
The unit map $\one\to\KO$ over any field with $p\neq 2$ induces a surjection
\[
  \pi_1\one\comp{2}\xrightarrow{e} K^M_1(k)/2\oplus \ZZ/2.
\]
In fact, the integral version of this result holds so that
\[
  \pi_1\one\xrightarrow{e} K^M_1(k)/2\oplus \ZZ/2
\]
is surjective.
\end{lemma}
\begin{proof}
First note that $e(\eta_s) = (0,1)$.  This fact is classical, coming
from the fact that $e(\eta_s)$ is represented by the matrix $(-1)$ in
$O(k)^{\operatorname{ab}}$ and the same topological-to-motivic
argument as in Lemma \ref{lemma:topmot} (with $\KO$ in place of $\KGL$).

Now note that there is a $GW(k)$-module map $GW(k)\to
k^M_1 \oplus \ZZ/2$ induced by
multiplying with $e(\eta_s)$.  This takes a one-dimensional quadratic
form $\langle u\rangle$ to $([u],1)$.  Since the unit map $\one\to \KO$
is a ring map, we have $e(\langle u\rangle \eta_s) = ([u],1)$.
Furthermore, $e((1+\langle u\rangle)\eta_s) = ([u],0)$, so we see that
$e$ is surjective.
\end{proof}

We can now write down the group structure on $\pi_1 \one\comp{2}$.

\begin{thm} \label{thm:2pi1}
Over a $2$-low-dimensional field $k$, $\pi_1\one\comp{2}$ fits into a short exact
sequence
\[
  0\to K^M_2(k)/8\to \pi_1\one\comp{2}\xrightarrow{e} k^M_1(k)\oplus \ZZ/2\to 0
\]
If we consider $\pi_{1+*\alpha}\one\comp{2}$ as a $\pi_{*\alpha}\one\comp{2}
\cong K^{MW}_{-*}{\comp{2}}$-module, then $K^M_2(k)/8 \subset
\pi_1\one\comp{2}$ consists of $K^{MW}_2(k)\comp{2}$-multiples of
$\nu$.  Moreover, the $K^{MW}_0(k)\comp{2}\cong GW(k)\comp{2}$-multiples of $\eta_s$ map onto
$K^M_1(k)/2 \oplus \ZZ/2$ via $\langle u\rangle \eta_s\mapsto ([u],1)$. 

As a short exact sequence of abelian groups, the $\ZZ/2\{\eta_s\}$-summand splits, 
while there is an
extension so that $[u]\eta\eta_s + [v]\eta\eta_s = [uv]\eta\eta_s - 4[u,v]\nu$.
\end{thm}
\begin{proof}
Lemma \ref{lemma:surj} gives us the surjectivity of $e$.  On the
$E_\infty$-page of the MANSS the kernel of $e$ is
represented by $\ZZ/4\{\alpha_{2/2}\}\otimes \pi_{-2\alpha}\M\ZZ_2$
and $\ZZ/2\{\alpha_1^3\}\otimes \pi_{1-3\alpha}\M\ZZ_2$.  The hidden
extension identified in Proposition \ref{prop:Kmod8} implies that the
kernel of $e$ is a copy of $K^M_2(k)/8$ generated by $\nu$.

To identify the group structure, first note that the $\ZZ/2$-summand
splits via $\varepsilon\mapsto \varepsilon\eta_s$ since $\eta_s$ has
order 2.  Thus it suffices to compute the addition
law on elements of $\pi_1$ represented by $[u]\eta\eta_s$.  Since $[uv] =
[u]+[v]+[u][v]\eta$ in $K^{MW}_1(k)$, we find that
\[
\begin{aligned}
  {}[u]\eta\eta_s + [v]\eta\eta_s
  &= ([uv]-[u][v]\eta)\eta\eta_s\\
  &= [uv]\eta\eta_s - [u][v]\eta^2\eta_s\\
  &= [uv]\eta\eta_s - 4[u,v]\nu,
\end{aligned}
\]
where the last equality follows from Proposition \ref{prop:4nu}.
\end{proof}

\section{The first rational stable motivic stem} \label{sec:rationalPi1}

In order to perform an arithmetic fracture computation of the first
stable motivic stem, we have to understand a couple of things about
$\one_\QQ$, the rationalization of the motivic sphere spectrum.  To
define $\one_\QQ$, consider a free resolution
\[
  \ZZ^{\oplus I}\xrightarrow{f} \ZZ^{\oplus J}\to \QQ\to 0.
\]
Then $\one_\QQ$ is the the cofiber
\[
  \bigvee_I \one\to \bigvee_J \one\to \one_\QQ
\]
of the map induced by $f$.  
As in topology,
the homotopy groups of $\EEE\smsh \one_\QQ$ for any spectrum $\EEE$ are
\[
  \pi_\star \EEE\smsh \one_\QQ \cong \pi_\star \EEE \otimes \QQ.
\]
We also note that
\[
  \one_\QQ \simeq \operatornamewithlimits{hocolim}(\one\xrightarrow{2}
  \one\xrightarrow{3} \one\xrightarrow{4} \one\xrightarrow{5}\cdots).
\]
(\emph{Warning}: 
$\one_\QQ$ is the rationalization of the motivic sphere spectrum over
$k$, not the motivic sphere spectrum over $\QQ$.)

Let $(-)_\QQ$ denote Bousfield localization at the rational Moore
spectrum $\one_\QQ$.

\begin{prop} \label{prop:Qsmashing}
For all motivic spectra $\EEE$,
\[
  \EEE_\QQ \simeq \EEE\smsh \one_\QQ
\]
and the natural map
\[
  \EEE\to \EEE\smsh \one_\QQ
\]
is the Bousfield localization of $\EEE$ at $\one_\QQ$.
\end{prop}
\begin{proof}
The proof is exactly as in \cite[Proposition 2.4]{MR551009}.
\end{proof}

Morel \cite{Morelsplitting}  
constructs the idempotent $\epsilon$ by twisting
$\GG_m\smsh \GG_m$ and proves that it corresponds to $-\langle
-1\rangle \in\pi_0\one = GW(k)$.  When $2$ is invertible, there are
corresponding projectors $e_+ = \frac{1}{2}(\epsilon-1)$ and $e_- =
\frac{1}{2}(\epsilon+1)$ and decompositions of spectra into $+$- and
$-$-components.  When $k$ is nonreal, we have $\epsilon = -1$, whence
$e_- = 0$ and $e_+$ is the identity map.  Combining the results of
\cite[\S16.2]{CD} with
\cite[\S5.3.35]{CD} proves that
$S_{\QQ+} \simeq \M\QQ$, so our hypothesis on $k$ implies that
$\one_\QQ\simeq \M\QQ$.  We record this fact in the following theorem.

\begin{thm}[Morel, Cisinski-Deglise] \label{thm:SQ}
If $k$ is nonreal (i.e., $-1$ is a sum of squares in $k$), 
then
\[
  \one_\QQ \simeq \M\QQ.
\]
In particular, if $k$ is low-dimensional with $p\ne 2$,
$\one_\QQ\simeq \M\QQ$.
\end{thm}

As such, it will be important to have vanishing results akin to Lemma
\ref{lemma:vanish} for $\pi_\star \M\QQ$.  The following lemma states
the known vanishing range without assuming the Beilinson-Soul\'{e}
vanishing conjecture.

\begin{lemma}\label{lemma:Qvanish}
For any field $k$, $\pi_{m+n\alpha}\M\QQ = 0$
whenever $n\ge -1$ and $m+n\alpha\ne 0$ or $-\alpha$.  We also have
$\pi_{m+n\alpha}\M\QQ = 0$ for $m<0$.
\end{lemma}
\begin{proof}
The first assertion follows from \cite{SV96} (see
also \cite[Corollary 4.2]{MVW}).  The second assertion is a consequence of
Morel's stable connectivity theorem \cite{Morelstableconnectivity}.
\end{proof}

\begin{rmk}\label{rmk:BS}
The rational version of Beilinson-Soul\'{e} vanishing asserts that
$\pi_{m+n\alpha}\M\QQ$ also vanishes whenever $m+n>0$.  In the low
range of dimensions we consider, these groups will not interfere with
our $\pi_\star \one$ computations, but they become important as soon
as one considers $\pi_{2+*\alpha}\one$.
\end{rmk}

\begin{prop} \label{prop:pi1SQ}
If $k$ is nonreal, then
\[
  \pi_{m+n\alpha} \one_\QQ = 0
\]
vanishes in the same range as $\pi_{m+n\alpha}\M\QQ$ stated in Lemma
\ref{lemma:Qvanish}.  In particular,
$\pi_{1+n\alpha}\one_\QQ = 0$ for $n\ge -1$.
\end{prop}
\begin{proof}
This is an immediate consequence of Theorem \ref{thm:SQ} and Lemma \ref{lemma:Qvanish}.
\end{proof}

\begin{prop} \label{prop:pi2SQ}
We have
\[
  \pi_i\left(\prod_\ell \one\comp{\ell}\right)_\QQ = 0
\]
over any low-dimensional field $k$, $i=1,2$.
\end{prop}
\begin{proof}
The result for $\pi_1$ is clear at this point.  By the methods of
Sections \ref{sec:largePrimePi1} and \ref{sec:2pi1}
we may compute the $2$-column of the $E_2$-page of the $\ell$-MANSS for all
primes $\ell$.  For each prime, all the entries are torsion, hence $\pi_2
\prod_\ell \one\comp{\ell}$ is torsion.  We conclude that $\pi_2$ of the
rationalization of $\prod_\ell \one\comp{\ell}$ is trivial.
\end{proof}

\section{The first stable motivic stem} \label{sec:pi1}
  
We employ the arithmetic fracture square of Appendix \ref{appendix} in order to glue together our MANSS and rational computations into an integral result.  
The homotopy pullback square
\[
\xymatrix{
  \one[1/p]\ar[r]\ar[d] &\prod_{\ell\ne p} \one\comp{\ell}\ar[d]\\
  \one_\QQ\ar[r] &\left(\prod_{\ell\ne p} \one\comp{\ell}\right)_\QQ
}
\]
induces the long exact sequence
\begin{equation}
\label{homotopyles}
\cdots\to \pi_2 \left(\prod_{\ell\ne p} \one\comp{\ell}\right)_\QQ\to \pi_1 \one[1/p] \to \pi_1
\one_\QQ\oplus \pi_1 \prod_{\ell\ne p}\one\comp{\ell} \to \pi_1
\left(\prod_{\ell\ne p}
  \one\comp{\ell}\right)_\QQ\to \cdots.
\end{equation}

By Propositions \ref{prop:pi1SQ} and \ref{prop:pi2SQ}, we see that,
over a low-dimensional field $k$, this results in an isomorphism 
\[
  \pi_1 \one[1/p]\to \pi_1 \prod_{\ell\ne p}\one\comp{\ell}.
\]

Theorems \ref{thm:p>3pi1}, \ref{thm:3pi1}, and \ref{thm:2pi1} now
imply our main theorem.

\begin{thm} \label{thm:pi1}
Over any low-dimensional field $k$ of exponential characteristic $p\ne 2,3$,
$\pi_1 \one[1/p]$ sits in a short exact sequence of abelian groups
\[
  0\to K^M_2(k)/24\to \pi_1\one[1/p]\to K^M_1(k)/2\oplus \ZZ/2\to 0.
\]
The subgroup $K^M_2(k)/24$ consists of $K^{MW}_2(k)$-multiples of $\nu$ and the $K^{MW}_0(k)$ $\cong
GW(k)$-multiples of $\eta_s$ map onto $K^M_1(k)/2\oplus \ZZ/2$ via
$\langle u\rangle \eta_s\mapsto ([u],1)$.  In general, the sequence
does not split.  Rather, the
$\ZZ/2\{\eta_s\}$-summand splits and there is an
extension so that $[u]\eta\eta_s + [v]\eta\eta_s = [uv]\eta\eta_s -
12[u,v]\nu$.

If $k$ is low-dimensional with $p=3$, then there is a short exact
sequence
\[
  0\to K^M_2(k)/8\to \pi_1\one[1/3]\to K^M_1(k)/2\oplus \ZZ/2\to 0
\]
with the same $K^{MW}_*(k)$-module structure and
\[
  [u]\eta\eta_s+[v]\eta\eta_s = [uv]\eta\eta_s - 4[u,v]\nu.
\]

If $k$ is low-dimensional with $p=2$, then
\[
  \pi_1\one[1/2] \cong K^M_2(k)/3,
\]
consisting of $K^{MW}_2(k)$-multiples of $\nu$.
\end{thm}
\begin{proof}
Assume $p\ne 2,3$.  The cases $p=2,3$ are similar (and easier) and we
leave those details to the interested reader.

We have already produced the short exact sequence for
$\pi_1\one[1/p]$.  Consider how $K^{MW}_2(k)$ acts on $\nu\in
\pi_{1+2\alpha}\one[1/p]$.  By our
low-dimensionality assumption, $K^{MW}_2(k)$ $\cong K^M_2(k)$.
Note that the $-2\alpha$-column of the $E_2$-page of the $\ell$-MANSS
consists only of $\ZZ_\ell\{1\}\otimes \pi_{-2\alpha}\MZ_\ell$ in
homological degree 0.  Computing the $1+2\alpha$-column of the MANSS
(see \S\ref{subsec:wt2}) and the multiplication on $E_\infty$-pages
implies that $K^{MW}_2(k)\cdot \nu$ is exactly the subgroup
$K^M_2(k)/24$ of $\pi_1\one[1/p]$.

Now consider how $K^{MW}_0(k)$ acts on $\eta_s$.  It suffices to
restrict our attention to the 2-MANSS where the $E_\infty$ associated
graded of $K^{MW}_0(k)$ is recorded in Table \ref{table:E2}.  Our
claim follows from recording how the 0-column acts on the 1-column
combined with the fact that $\eta_s$ is the generator of
$\ZZ/2\{\alpha_1\}\otimes \pi_{1-\alpha}\MZ_2$ or
$\Tor(\ZZ/2\{\alpha_1\},\pi_{-\alpha}\MZ_2)$ (only one of which is nonzero).

It remains to calculate the value of $[u]\eta\eta_s + [v]\eta\eta_s$ in
$\pi_1\one[1/p]$.  Under the isomorphism 
$\pi_1 \one[1/p]\to \pi_1 \one\comp{2}\oplus K^M_2(k)/3$ we
have
\[
  [u]\eta\eta_s+[v]\eta\eta_s-[uv]\eta\eta_s\mapsto (-4[u,v]\nu,0).
\]
Since $12[u,v]\nu$ has the same image, we can conclude that
\[
  [u]\eta\eta_s + [v]\eta\eta_s = [uv]\eta\eta_s-12[u,v]\nu.
\]
\end{proof}

\section{Nonzero weights} \label{sec:nonzero}


In this section we briefly record how the computations play out in
weights other than $0$.  Throughout we assume that $k$ is a
low-dimensional field and use Theorem \ref{thm:UCT} as our primary
computational tool.  We warn the reader that we have included much
less detail in this section and the following should be viewed only as
a roadmap for verifying these calculations.


\subsection{Weights greater than $2$}

For an integer $n> 2$ it is easy to see that the $1+n\alpha$-column of
the $\ell$-MANSS is trivial for $\ell\ne 2$.  When $\ell = 2$ (and
$p\ne 2$) and $n\ge 5$, the $n\alpha$-, $1+n\alpha$-, and $2+n\alpha$-columns take
the form depicted in Table \ref{table:wt5MANSS}.

\begin{table}
\renewcommand{\arraystretch}{1.5}
\begin{tabular}[h]{|c|c|c|}
\vdots&\vdots&\vdots\\\hline

0&0&0\\\hline

$\ZZ/2\{\alpha_1^{n+2}\}\otimes \pi_{-2\alpha}\MZ_2$&
$\ZZ/2\{\alpha_1^{n+3}\}\otimes \pi_{1-3\alpha}\MZ_2$& 0\\\hline

$\ZZ/2\{\alpha_1^{n+1}\}\otimes \pi_{-\alpha}\MZ_2$&
$\ZZ/2\{\alpha_1^{n+2}\}\otimes \pi_{1-2\alpha}\MZ_2$& 0\\\hline

& $\ZZ/2\{\alpha_1^{n+1}\}\otimes \pi_{1-\alpha}\MZ_2$&\\
$\ZZ/2\{\alpha_1^n\}\otimes \pi_0\MZ_2$& $\oplus$& 0\\
& $\Tor(\ZZ/2\{\alpha_1^{n+2}\},\pi_{-2\alpha}\MZ_2)$&\\\hline

0& $\Tor(\ZZ/2\{\alpha_1^{n+1}\},\pi_{-\alpha}\MZ_2)$&
$\ZZ/2\{\alpha_3\alpha_1^{n-1}\}\otimes \pi_{-2\alpha}\MZ_2$\\\hline

0& 0& $\ZZ/2\{\alpha_3\alpha_1^{n-2}\}\otimes
\pi_{-\alpha}\MZ_2$\\\hline

0& 0& $\ZZ/2\{\alpha_3\alpha_1^{n-3}\}\otimes \pi_0\MZ_2$\\\hline
0&0&0\\\hline
$n\alpha$&$1+n\alpha$&$2+n\alpha$
\end{tabular}
\renewcommand{\arraystretch}{1.0}
\caption{The first three $*+n\alpha$-columns of the MANSS over a
  low-dimensional field for $n\ge 5$}\label{table:wt5MANSS}
\end{table}

For $n=3,4$ there are additional terms arising from $\alpha_{2/2}^2 =
\beta_{2/2}$ in the $2+n\alpha$-column, but they aren't important for
our argument.

A pattern of $d_2$- and $d_3$-differentials on the terms in the
$2+n\alpha$-column kill off all of the $1+n\alpha$-column.  The point
is that the $n\alpha$-column is permanent, being the associated graded
of $K^{MW}_{-n}(k)\comp{2} \cong W(k)\comp{2}$.  Meanwhile, every
element in the $1+n\alpha$-column represents a multiple
of $\eta^3\eta_s = 4\eta\nu = 0$ and hence is hit by a differential.
(This is of course all very plausible as $K^M_{r-1}(k)/2 = \ZZ/2\otimes
\pi_{-(r-1)\alpha}\MZ_2$ is an extension of
$\Tor(\ZZ/2,\pi_{-r\alpha}\MZ_2)$ and $\ZZ/2\otimes
\pi_{1-r\alpha}\MZ_2$.)  We
conclude that $\pi_{1+n\alpha}\one\comp{2} = 0$ for $n\ge 3$.

In fact,  
$\pi_{1+n\alpha}\one[1/p] = 0$ for $n\ge 3$.  By the
arithmetic fracture square and the fact that the $2+n\alpha$-column is
torsion in every $\ell$-MANSS, it suffices to check that
$\pi_{1+n\alpha}\one_\QQ = 0$, which is a special case of Proposition
\ref{prop:pi1SQ}.

\begin{rmk}
It is actually the case that the $d_2$- and $d_3$-differentials responsible
for the vanishing of $\pi_{1+n\alpha}\one\comp{2}$ for $n\ge 3$
eliminate all the classes in the $2+n\alpha$-column of the 2-MANSS for
$n\ge 5$, implying that $\pi_{2+n\alpha}\one\comp{2} = 0$ for $n\ge
5$.  There are no contributions at other primes, so get the intriguing vanishing
result
\begin{equation}\label{eqn:pi2vanish}
  \pi_{2+n\alpha}\one[1/p] = 0\text{ for }n\ge 5
\end{equation}
over a low-dimensional field.  It would be interesting to pursue other
vanishing results in the $m+*\alpha$-lines via our methods.

With a great deal of perseverance one can probably use
our methods to determine the
$2+*\alpha$-line of coefficients for $\one[1/p]$.  We have gone so far
as to determine the permanent cycles in these columns of the
$\ell$-MANSS, but have thus far not been able to determine all the
differentials entering from the $3+*\alpha$-columns.
\end{rmk}


\subsection{Weight $2$} \label{subsec:wt2}
This weight has no $\ell$-adic contributions for $\ell\ne 2,3$.  The
$1+2\alpha$-column of the 3-MANSS consists of $\ZZ/3\{\alpha_1\}$ in
homological degree $1$, so $\pi_{1+2\alpha}\one\comp{3} = \ZZ/3$.

The $1+2\alpha$-column of the $E_2$-page of the 2-MANSS in weight $2$ is similar to the weight 3
picture above:  one simply decreases all the $\alpha_1$ exponents by
$1$ and adds in one new group, $\ZZ/4\{\alpha_{2/2}\}\otimes \pi_0\MZ_2$.
These terms combined with 
\[
  \ZZ/2\{\alpha_1^3\}\otimes \pi_{1-\alpha}\MZ_2\text{ and }\Tor(\ZZ/2\{\alpha_1^3\},\pi_{-\alpha}\MZ_2)
\]
assemble into
$\pi_{1-\alpha}\M\FF_2$ and there is an
extension so that $4\alpha_{2/2}$ equals the generator of this group, thus
producing $\ZZ/8\{\nu\}$ in $\pi_{1+2\alpha}\one\comp{2}$.  Everything
else dies off via a pattern on $d_2$ and $d_3$ differentials because
$\eta^3\eta_s = 0$, and we
get $\pi_{1+2\alpha}\one\comp{2}\cong \ZZ/8\{\nu\}$.

Again by Proposition \ref{prop:pi1SQ} and arithmetic fracture we can
conclude that
\[
  \pi_{1+2\alpha}\one[1/p] \cong \ZZ/24\{\nu\}.
\]

\subsection{Weight $1$}\label{subsec:wt1}

By very similar arguments, we have
\[
  \pi_{1+\alpha}\one[1/p] \cong K^M_1(k)/24\{\nu\}\oplus \ZZ/2\{\eta\eta_s\}.
\]

\subsection{Weight $-1$}

This weight has no $\ell$-adic contributions for $\ell\ne 2$.  The
$1-\alpha$-column in the $2$-MANSS has entries
\[
  \pi_{1-\alpha}\MZ_2\{1\},~\Tor(\ZZ/2\{\alpha_1\},\pi_{-2\alpha}\MZ_2),
\]
\[
  \ZZ/2\{\alpha_1\}\otimes \pi_{1-2\alpha}\MZ_2,\text{ and }\ZZ/2\{\alpha_1^2\}\otimes\pi_{1-3\alpha}\MZ_2
\]
in homological degrees $0$, $0$, $1$, and $2$, respectively.  We
deduce that $\pi_{1-\alpha}\one\comp{2}$ has an associated graded with
pieces $\pi_{1-\alpha}\MZ_2\{1\}$, $K^M_1(k)/2\{\eta_s\}$,
$K^M_2(k)/2\{\eta\eta_s\}$.
Again by Proposition \ref{prop:pi1SQ} and arithmetic fracture we have
\[
  \pi_{1-\alpha}\one[1/p]\cong K^M_2(k)/2\oplus K^M_1(k)/2.
\]

\subsection{Contributing terms}
\label{subsection:contributingterms}
In weights $n=-2,-3,-4$ our methods no longer offer control over integral computations of $\pi_{1+n\alpha}\one[1/p]$.
Indeed, 
the terms 
\[
\pi_{2+n\alpha}\left(\prod_{\ell\ne p}\one\comp{\ell}\right)_\QQ,  
\pi_{1+n\alpha}\left(\prod_{\ell\ne p}\one\comp{\ell}\right)_\QQ 
\text{ and }
\pi_{1+n\alpha} \one_\QQ
\]
from (\ref{homotopyles}) can be nontrivial in general.
We analyze the rational stable motivic homotopy group in some detail.
Theorem \ref{thm:SQ} implies 
\[
\pi_{1+n\alpha} \one_\QQ
\cong
H^{-1-n}(k;\QQ(-n)).
\]
For $n=-2$, 
$H^{1}(k;\QQ(2)) \cong K_{3}^{ind}(k)_{\QQ}$, 
the rationalized indecomposable $K_{3}$-group of $k$.
If $k$ is a finite field or a global field of positive characteristic this group is trivial.  
If $k$ is a nonreal number field,
then $H^{1}(k;\QQ(2)) = \QQ^{r_2}$, 
where $r_2$ is the number of pairs of complex valuations.
The papers \cite{LevineK3}, \cite{Merkurjev-Suslin} give elegant proofs of these results.
For $n=-3$,
the rational motivic cohomology group $H^{2}(k;\QQ(3))$ is a direct summand of $K_{4}(k)_{\QQ}$,
which is trivial for finite and global fields.
For $n=-4$, 
the group $H^{3}(k;\QQ(4))$ is trivial for every low-dimensional field
$k$ since $\QQ$ is divisible \cite[I \S3.1 Corollary]{Serre}.

The terms $\pi_{i+n\alpha}\left(\prod_{\ell\ne p}
  \one\comp{\ell}\right)_\QQ$ for $i=1,2$ and $n=-2,-3,-4$ are of the
form
\[
  \pi_{i+n\alpha}\left(\prod_{\ell\ne p}
  \one\comp{\ell}\right)_\QQ = \prod_{\ell\ne
  p}H^{-i-n}(k;\ZZ_\ell(-n))\otimes \QQ.
\]
If $i=1$ and $n=-4$, then this group vanishes, but in general and it
can be nontrivial for low-dimensional fields.

\subsection{Weight $-2$}
This weight has no $\ell$-adic contributions for $\ell\ne 2$.  
The $1-2\alpha$-column in the $2$-MANSS has entries $\pi_{1-2\alpha}\MZ_2\{1\}$ and 
$\ZZ/2\{\alpha_1\}\otimes\pi_{1-3\alpha}\MZ_2$ in homological degrees $0$ and $1$,
respectively.  
We deduce that $\pi_{1-2\alpha}\one\comp{2}$ has an associated graded with pieces $\pi_{1-2\alpha}\MZ_2\{1\}$ and $K^M_2(k)/2$.
Hence there is short exact sequence
\[
  0\to K^M_2(k)/2\to \pi_{1-2\alpha}\prod_{\ell\ne p}\one\comp{\ell} \to
  \pi_{1-2\alpha}\MZ_2\to 0.
\]
From (\ref{homotopyles}) we obtain the exact sequence
\begin{equation*}
\pi_{2-2\alpha} \left(\prod_{\ell\ne p} \one\comp{\ell}\right)_\QQ
\to 
\pi_{1-2\alpha}  \one[1/p] 
\to 
K_{3}^{ind}(k)_{\QQ}\oplus \pi_{1-2\alpha} \prod_{\ell\ne p}\one\comp{\ell} 
\to 
\pi_{1-2\alpha} 
\left(\prod_{\ell\ne p}
\one\comp{\ell}\right)_\QQ.
\end{equation*}

\subsection{Weight $-3$}
Since there is only one $2$-adic contribution of homological degree $0$,  
we find that
\[
  \pi_{1-3\alpha}\prod_{\ell\ne p}\one\comp{\ell}\cong \pi_{1-3\alpha}\MZ_2\{1\}.
\]
From (\ref{homotopyles}) we obtain the exact sequence
\begin{equation*}
\pi_{2-3\alpha} \left(\prod_{\ell\ne p} \one\comp{\ell}\right)_\QQ
\to 
\pi_{1-3\alpha}  \one[1/p] 
\to 
H^{2}(k;\QQ(3))\oplus H^{2}(k;\ZZ_{2}(3))
\to 
H^{2}(k;\QQ_{2}(3)).
\end{equation*}
This sequence simplifies considerably when $k$ is a finite or global field, 
cf.~\S\ref{subsection:contributingterms}.

\subsection{Weight $-4$}

In this weight there are no $\ell$-adic contributions, 
i.e.,
\[
\pi_{1-4\alpha}\prod_{\ell\ne p}\one\comp{\ell} = 0.
\]
Thus, 
from (\ref{homotopyles}) and \S\ref{subsection:contributingterms}, 
$\pi_{2-4\alpha}\left(\prod_{\ell\ne p}\one\comp{\ell}\right)_\QQ$ surjects onto $\pi_{1-4\alpha}\one[1/p]$.

\subsection{Weights less than $-4$}

By the low-dimensionality assumption, we see that the
$1+n\alpha$-column of the $\ell$-MANSS vanishes whenever $\ell\ne p$
and $n<-4$.  Moreover, in this range $\pi_{i+n\alpha}\left(\prod_{\ell\ne p}\one\comp{\ell}\right)_\QQ$ vanishes for $i=1,2$.  
It follows that for $n<-4$,
$\pi_{1+n\alpha}\one[1/p]\cong \pi_{1+n\alpha}\one_\QQ\cong H^{-1-n}(k;\QQ(-n))=0$, 
cf.~\S\ref{subsection:contributingterms}.



\subsection{Finite fields}
When $k$ is a finite field we obtain a complete computation of the $1$-line $\pi_{1+n\alpha}\one[1/p]$ from the above.
For example, assuming $p\neq 2$, we obtain
\begin{equation*}
\pi_{1+n\alpha}\one[1/p]
=
\begin{cases}
0 & n\ge 3 \\
\ZZ/24 & n=2 \\
k^{\times}/24\oplus\ZZ/2 & n=1 \\
\ZZ/2\oplus\ZZ/2 & n=0 \\
k^{\times}/2 & n=-1 \\
0 & n\le -2.
\end{cases}
\end{equation*}



\appendix
\section{Arithmetic fracture squares in the stable motivic homotopy category}\label{appendix}

To a motivic spectrum $\EEE$ we can associate its rationalization $\EEE_{\QQ}$ and $\ell$-adic completion $\EEE\comp{\ell}$ 
for each rational prime number $\ell$.
We claim that there is a homotopy cartesian square: 
\begin{equation}
\label{equation:arithmeticsquare1}
\xymatrix{ 
\EEE\ar[r] \ar[r] \ar[d] & \prod_{\ell}\EEE\comp{\ell} \ar[d]  \\
\EEE_{\QQ} \ar[r] & (\prod_{\ell}\EEE\comp{\ell})_{\QQ}   }
\end{equation}

A theory of classical Bousfield localization for motivic spectra is worked out in \cite{MR2399164}.
When localizing at a motivic spectrum $\EEE$ the idea is to associate to any motivic spectrum $\FFF$ the part of $\FFF$
that can be seen through the eyes of $\EEE$.
More precisely,
there exists an $\EEE$-equivalence $\eta_{\EEE}\colon\FFF\rightarrow L_{\EEE}\FFF$ where $L_{\EEE}\FFF$ is $\EEE$-local.
Recall that a map $\FFF\rightarrow\GGG$ is called an $\EEE$-equivalence if smashing with $\EEE$ yields a weak equivalence.
Moreover, 
$\GGG$ is $\EEE$-local if the group of stable homotopy classes of maps $[\FFF,\GGG]$ is trivial
for every $\FFF$ with the property that $\EEE\wedge\FFF\simeq\ast$,
i.e., $\FFF$ is $\EEE$-acyclic.
Let $\one/\ell$ denote the mod-$\ell$ Moore spectrum.
Then $L_{\one/\ell}\FFF$ is the $\ell$-adic completion $\FFF\comp{\ell}$ of $\FFF$ \cite[\S 3]{MR2399164}.
And the rationalization $\FFF_{\QQ}$ of $\FFF$ is the localization $L_{\one_{\QQ}}\FFF$ with respect to the rationalized 
motivic sphere spectrum.

\begin{thm}
\label{thm:square}
Suppose $\EEE$, $\FFF$ and $\GGG$ are motivic spectra and
$\EEE \smsh L_{\FFF}\GGG\simeq\EEE \smsh L_{\FFF}L_{\EEE}\GGG\simeq\ast$.  
Then there is a homotopy cartesian square:
\begin{equation}
\label{equation:arithmeticsquare2}
\xymatrix{ 
L_{\EEE\vee\FFF}\GGG\ar[r]^-{\eta_{\EEE}} \ar[d]_-{\eta_{\FFF}} & L_{\EEE}\GGG \ar[d]^-{\eta_{\FFF}}  \\
L_{\FFF}\GGG \ar[r]^-{L_{\FFF}\eta_{\EEE}} & L_{\FFF}L_{\EEE}\GGG   }
\end{equation}
\end{thm}

Before starting the proof of Theorem \ref{thm:square} proper we deduce our main example (\ref{equation:arithmeticsquare1}) 
of an arithmetic square.
When $\EEE$ is the sum $\bigvee_{\ell}\one/\ell$ indexed over all rational prime numbers, 
there are no nontrivial homotopy classes of maps from an $\EEE$-acyclic space to a product of spectra of the form 
$\prod_{\ell}\GGG\comp{\ell}$.
Moreover,
since smashing with $\one/\ell$ commutes with products,
there is a natural isomorphism
\begin{equation*}
\one/\ell_{\star}(\GGG)
\rightarrow
\one/\ell_{\star}(\prod_{\ell}\GGG\comp{\ell}).
\end{equation*}
It follows that $L_{\EEE}\GGG=\prod_{\ell}\GGG\comp{\ell}$.
In addition, 
when $\FFF=\one_{\QQ}$ then $\EEE \smsh L_{\FFF}\GGG \simeq \EEE \smsh
\left(\prod_\ell \GGG\comp{\ell}\right)_\QQ \simeq *$, 
since $\one_{\QQ}\wedge\one/\ell\simeq\ast$ for every $\ell$.
This shows the arithmetic square (\ref{equation:arithmeticsquare1}) for $\EEE$,
and hence for $\one[1/p]$ as in Section \ref{sec:pi1}, 
is a special case of (\ref{equation:arithmeticsquare2}).

\begin{proof}[Proof of Theorem \ref{thm:square}]
In the diagram (\ref{equation:arithmeticsquare2}), 
the $\EEE$-equivalence $\eta_{\EEE}$ is the unique factorization of $\eta_{\EEE}\colon\GGG\rightarrow L_{\EEE}\GGG$ through 
$L_{\EEE\vee\FFF}\GGG$.
It exists because $\GGG\rightarrow L_{\EEE\vee\FFF}\GGG$ is an $\EEE$-equivalence.
The same remarks apply to the map $\eta_{\FFF}$.
Let $\hp$ denote the homotopy pullback of 
\begin{equation}
\label{equation:homotopypullback}
\xymatrix{ 
L_{\FFF}\GGG \ar[r]^-{L_{\FFF}\eta_{\EEE}} & L_{\FFF}L_{\EEE}\GGG & L_{\EEE}\GGG \ar[l]_-{\eta_{\FFF}}.  }
\end{equation} 
Suppose $\HHH$ is both $\EEE$- and $\FFF$-acyclic.
Then the long exact sequence
\begin{equation*}
\cdots
\rightarrow
[\HHH,\hp]
\rightarrow
[\HHH,L_{\EEE}\GGG]
\oplus
[\HHH,L_{\FFF}\GGG]
\rightarrow
[\HHH,L_{\FFF}L_{\EEE}\GGG]
\rightarrow
\cdots
\end{equation*}
obtained from (\ref{equation:homotopypullback}) implies $[\HHH,\hp]=0$.
Thus $\hp$ is $(\EEE\vee\FFF)$-local.
\vspace{0.1in}

It remains to show the induced map $\GGG\rightarrow\hp$ is both an $\EEE$- and $\FFF$-equivalence.
To begin, 
note that $\hp\rightarrow L_{\EEE}\GGG$ is an $\EEE$-equivalence, 
being the pullback of $L_{\FFF}\eta_{\EEE}$.
Thus since $\GGG\rightarrow L_{\EEE}\GGG$ is an $\EEE$-equivalence, 
so is $\GGG\rightarrow\hp$.
By pulling back $\eta_{\FFF}$ a verbatim argument shows that $\GGG\rightarrow\hp$ is an $\FFF$-equivalence. 
\end{proof}

\begin{cor}
A map between motivic spectra $\EEE\rightarrow\FFF$ is a weak equivalence if and only if the naturally induced 
maps $L_{\QQ}\EEE\rightarrow L_{\QQ}\FFF$ and $\EEE\wedge\one/\ell\rightarrow\FFF\wedge\one/\ell$ are weak equivalences 
for every rational prime number $\ell$.
\end{cor}
\begin{proof}
Apply Theorem \ref{thm:square} in the case when $\EEE=\bigvee_{\ell}\one/\ell$ and $\FFF=\one_{\QQ}$.
\end{proof}

It is curious to note that one can also produce a direct proof that
(\ref{equation:arithmeticsquare1}) is homotopy cartesian by employing generators of 
the motivic stable homotopy category $\SH$ and comparing with the
classical stable homotopy category $\SH^{\TOP}$.  In fact, the
following proof works in any model category with a set of compact
generators that is also enriched over topological spectra.

\begin{proof}[Alternate verification that
  (\ref{equation:arithmeticsquare1}) is homotopy cartesian.]
Suppose $X$ is a smooth scheme of finite type over the base scheme in $\SH$.
For integers $m$ and $n$ we form the smash product $X^{m+n\alpha}\equiv S^{m+n\alpha}\wedge X_{+}$ in $\SH$.
Then $\{X^{m+n\alpha}\}_{m,n\in\ZZ}$ is a set of compact generators of $\SH$.
\vspace{0.1in}

Denote by 
\begin{equation*}
\uhom
\colon
\SH^{\op}\times\SH
\rightarrow
\SH^{\TOP}
\end{equation*}
the composite of the internal hom functor in $\SH$ with the right adjoint of the canonical functor
\begin{equation*}
\SH^{\TOP}
\rightarrow
\SH^{S^{1}}
\rightarrow
\SH.
\end{equation*}
It is obtained by first viewing ordinary spectra as constant presheaves of $S^{1}$-spectra for the simplicial 
circle and second as motivic spectra.
We claim that applying $\uhom(X^{m+n\alpha},-)$ to (\ref{equation:arithmeticsquare1}) yields the classical arithmetic 
square of $\uhom(X^{m+n\alpha},\EEE)$ shown in \cite[\S 2]{MR551009}. 
This finishes the proof.
In effect, 
$\uhom(X^{m+n\alpha},-)$ commutes with products and with $\ell$-completions because the latter is an internal 
hom \cite[\S 3]{MR2399164}.
Moreover, 
since the generator $X^{m+n\alpha}$ is compact,
$\uhom(X^{m+n\alpha},-)$ is seen to commute with rationalization by viewing the latter as a filtered homotopy colimit. 
\end{proof}

\bibliographystyle{plain} 
\bibliography{OO}
\vspace{0.1in}

\end{document}